\documentclass[11pt]{amsart}
\usepackage{amsaddr}

\usepackage[english]{babel}
\usepackage{amsmath, amsfonts, amssymb, amsthm, epsfig, graphicx, url, hyperref, color, tikz, cancel,comment,mathtools}
\mathtoolsset{showonlyrefs}

\newtheorem{theorem}{Theorem}[section]
\newtheorem{lemma}[theorem]{Lemma}
\newtheorem{corollary}[theorem]{Corollary}
\newtheorem{proposition}[theorem]{Proposition}
\newtheorem{definition}[theorem]{Definition}

\newtheorem{remark}[theorem]{Remark}

\newcounter{theor}

\usepackage[sort,compress]{cite}

\makeatletter
\newtheorem*{rep@theorem}{\rep@title}
\newcommand{\newreptheorem}[2]{%
\newenvironment{rep#1}[1]{%
 \def\rep@title{#2 \ref{##1}}%
 \begin{rep@theorem}}%
 {\end{rep@theorem}}}
\makeatother
\newreptheorem{theorem}{Theorem}
\newreptheorem{corollary}{Corollary}

\def\conv{\mathop\mathrm{conv}\nolimits}

\def\supp{\mathop\mathrm{supp}\nolimits}
\def\B{\mathbb{B}}
\def\s{\mathbb{S}}

\def\S{\mathcal{S}}

\def\R{\mathbb{R}}
\def\N{\mathbb{N}}

\def\vol{\mathrm{vol}}

\newcommand{\dlat}{\mathrm{d}}

\newcommand{\Mel}[2]{\mathcal{M}_{#1}(#2)}

\def \Rn{{\R^n}}
\newcommand{\hid}{m}

\def \B{B_2^n}

\def \s{\mathbb{S}^{n-1}}
\def \S{\mathbb{S}^{n\hid-1}}

\def \pp{\Pi^{\circ}}

\newcommand {\LYZP}[1]{\PP \langle#1\rangle}
\newcommand {\lc}{\operatorname{LC}_n}

\def \conv{\operatorname{conv}}

\def \PP{\Pi^{\circ,\hid}}

\def \Cov{g_{K,\hid}}

\numberwithin{equation}{section}

\begin{document}
\title[Functional-$m$th-order inequalities]{Higher-Order Reverse Isoperimetric Inequalities for Log-concave Functions}

\author[D. Langharst]{Dylan Langharst}
\address{Carnegie Mellon University\\ Department of Mathematical Sciences \\ Pittsburgh, PA 15213, USA\\ OrcID: 0000-0002-4767-3371}
\email{dlanghar@andrew.cmu.edu}

\author[F. Mar\'in]{Francisco Mar\'in Sola}
\address{Departamento de Ciencias \\ Centro Universitario de la Defensa (CUD) \\ 30729 San Javier (Murcia), Spain \\ OrcID: 0000-0003-1538-7394}
\email{francisco.marin7@um.es} 

\author[J. Ulivelli]{Jacopo Ulivelli}
\address{Institut f\"ur Diskrete Mathematik und Geometrie \\ Technische Universit\"at Wien \\ Wiedner Hauptstraße 8-10, E104-6
\\ 1040 Wien, Austria\\ OrcID: 0000-0002-4726-5271}
\email{jacopo.ulivelli@tuwien.ac.at}



\begin{abstract}
The Rogers-Shephard and Zhang's projection inequalities are two reverse, affine isoperimetric-type inequalities for convex bodies. Following a classical work by Schneider, both inequalities have been extended to the so-called $m$th-order setting. In this work, we establish the $m$th-order analogues for these inequalities in the setting of log-concave functions. Our proof of the functional Zhang's projection inequality employs properties of the asymmetric LYZ body, significantly streamlining the argument and producing a novel approach for the case $m=1$. Furthermore, we introduce and analyze the radial mean bodies of a log-concave function, thereby providing a functional generalization of Gardner and Zhang's radial mean bodies. These are new even in the case $m=1$. Our development leverages an extension of Ball bodies, which may be of independent interest.

\end{abstract}

\maketitle

\section{Introduction}

In the Brunn-Minkowski theory of convex bodies (i.e., compact, convex subsets of $\R^n$ with non-empty interior, for $n \in \N$ fixed), two main affinely invariant inequalities concern the difference body and the polar projection body of a convex body $K$ in $\R^n$. The starting points of this work are the two corresponding reverse affine isoperimetric inequalities: the Rogers-Shephard and Zhang's projection inequalities, which we will introduce below; see Yang's survey \cite{Y10} for an introduction to the theory of affine structures in convex geometry. Recall that the Minkowski sum of two sets $A, B  \subseteq \R^n$ is the set $$A+B=\{a+b:a\in A, b\in B\}.$$ 
We say that $K$ is symmetric about the origin if $K=-K$ and that $K$ is symmetric if a translate of $K$ is symmetric about the origin. Every origin-symmetric convex body is the unit ball of a norm. More generally, given a convex body $K$ containing the origin, it is the unit ball of its Minkowski functional, or gauge: for $x\in \R^n$, $\|x\|_K=\inf\{t>0: x\in t K\}.$ Additionally, the support function of a convex body $K$ is precisely $h_K(x)=\sup_{y\in K}\langle y,x\rangle.$

\subsection{Affine isoperimetric-type inequalities and their higher-order extensions.}
For a convex body $K$, its difference body is defined as $$DK\coloneqq K+(-K).$$ The associated isoperimetric-type inequalities relating $DK$ and $K$ are
\begin{equation}
    \label{eq:RS}
    2^n \leq \frac{\vol_n(DK)}{\vol_n(K)}\leq \binom{2n}{n},
\end{equation}
where $\vol_n$ is the Lebesgue measure on $\R^n$. The left-hand side of \eqref{eq:RS}, which follows from the Brunn-Minkowski inequality (see \eqref{e:BM} later), becomes an equality if and only if $K$ is symmetric. The right-hand side of \eqref{eq:RS} is known as the Rogers-Shephard inequality. Here, equality holds if and only if $K$ is an $n$-dimensional simplex \cite{RS57}. We recall that the convex hull of a set $A \subset \R^n$, denoted $\conv A$, is the smallest convex set containing $A$; an $n$-dimensional simplex is the convex hull of $(n+1)$ affinely independent points.

Denoting $\vol_{n-1}(P_{x^\perp} K)$ as the volume of the orthogonal projection of $K$ onto the linear subspace orthogonal to $x\in\R^n\setminus\{o\}$, the polar projection body $\pp K$ of a convex body $K$ is the unit ball of the norm
\begin{equation}\|x\|_{\pp K} \coloneqq \frac{1}{2}\int_{\partial K}|\langle x,n_K(y) \rangle|\,\dlat \mathcal{H}^{n-1}(y) = |x|\vol_{n-1}(P_{x^\perp} K),
\label{eq:pp}
\end{equation}
where $\mathcal{H}^{n-1}$ is the $(n-1)$-dimensional Hausdorff measure on $\partial K$, the boundary of $K$, and $n_K:\partial K \longrightarrow \s$ is the generalized Gauss map. The latter is a well-defined function $\mathcal{H}^{n-1}$-almost everywhere. On the subset of $\partial K$ in which it is well-defined, it associates each point with the corresponding unique outer-unit normal vector.

The following inequalities then hold, denoting the unit Euclidean ball in $\R^n$ as $\B$: \begin{equation}
\label{eq:zhang}
\frac{1}{n^n}\binom{2n}{n}\leq \frac{\vol_n(\pp K)}{\vol_n(K)^{1-n}}\leq \left(\frac{\vol_n(B_2^n)}{\vol_{n-1}(B_2^{n-1})}\right)^n.\end{equation}
The left-hand side is Zhang's projection inequality, where equality holds if and only if $K$ is an $n$-dimensional simplex \cite{Zhang91}. The right-hand side is Petty's projection inequality, with equality if and only if $K$ is an ellipsoid \cite{CMP71}.

An essential link between the difference body $DK$ and polar projection body $\Pi^\circ K$ is provided by the covariogram. For a convex body $K$, the covariogram is defined as
\begin{equation}\label{eq:covariogram}
g_K(x)=\vol_n(K\cap (K+x));
\end{equation} 
see the recent survey by Bianchi \cite{GB23} for a rich overview of this function. We merely mention that, $g_K$ is $(1/n)$-concave via an application of the Brunn-Minkowski inequality \eqref{e:BM}, that the support of $g_K$ is $DK$, and that \begin{equation}\label{e:matheron}
    \frac{\partial}{\partial r}g_K(r\theta)\big|_{r=0^+}=-\|\theta\|_{\pp K}, \quad \theta\in\s,
\end{equation}
as shown by Matheron \cite{MA}.  Note also that $\max_{x \in \R^n} g_K(x)=g_K(o)=\vol_n(K)$. Moreover, observing that $$\int_{\R^n}g_K(x)\,\dlat x = \vol_n(K)^2,$$
one may combine the Rogers-Shephard and Zhang's projection inequalities into
\begin{equation}\vol_n(DK)\vol_n(K)\!\leq\! \binom{2n}{n}\int_{\R^n}g_K(x)\,\dlat x\leq n^n\vol_n(K)^{n+1}\vol_n(\pp K).
\label{eq:Zhang_other_form}
\end{equation}

We now introduce the first framework for the extensions we consider. We write $\bar x=(x_1,\dots,x_\hid)$, $x_i\in\R^n,$ for vectors in $(\R^n)^\hid \simeq \R^{n\hid}$. For a convex body in $\R^n$, Schneider \cite{Sch70} introduced the difference body of order $m\in\N$,
\begin{equation}
\label{eq:mDiff}
D^\hid(K) \coloneqq \left\{\overline{x}=(x_1,\dots,x_{\hid}) \in \R^{n\hid} \colon K \bigcap_{i=1}^\hid (x_i+K) \neq \emptyset \right\} \subset \R^{n\hid},
\end{equation}
and then established the following $m$th-order Rogers-Shephard inequality
\begin{equation}
    \frac{\vol_{n\hid}\left(D^\hid (K)\right)}{\vol_n(K)^{\hid}}\leq \binom{n(\hid +1)}{n}.
    \label{eq:RSell}
\end{equation}
Equality in \eqref{eq:RSell} holds if and only if $K$ is an $n$-dimensional simplex.
Fixing $m\in\N$, Schneider's primary tool \cite{Sch70} is an extension of the covariogram function of a convex body $K$, namely
\begin{equation}
    \Cov(\bar x)=\vol_n\left(K \bigcap_{i=1}^\hid (x_i+K)\right),
    \label{eq:covario_hi}
\end{equation}
which is supported on $D^{\hid}(K)$. It follows from the Brunn-Minkowski inequality that $\Cov(\bar x)^\frac{1}{n}$ is concave for all $\hid\in\N$ (see e.g. \cite[Lemma 5.1]{Sch70} for a proof).

Concerning a lower bound to \eqref{eq:RSell}, Schneider showed for $n=2$ and for any $\hid\in\mathbb{N}$ that the minimum is obtained for every symmetric convex body. For $n\geq 3$ and $\hid\geq 2,$ this is instead false. In the same work, Schneider conjectured that \eqref{eq:RSell} is minimized by ellipsoids for such $n$ and $\hid$.

The interest in the construction \eqref{eq:mDiff} has been recently reignited by the work of Roysdon \cite{Ro20}, where \eqref{eq:RSell} was generalized to product measures with suitable concavity properties. This setting was further investigated by the first author and collaborators in \cite{HLPRY25}, where \eqref{eq:mDiff} served as a starting point for a wider $m$th-order setting. 

In \cite{HLPRY25}, the following notion was introduced. For a convex body $K$ in $\R^n$, its $m$th-order polar projection body $\PP K\subset \R^{nm}$ is given by the gauge
\begin{equation}
\label{eq:PP}
\|\bar\theta\|_{\PP K} = \!\int_{\partial K}h_{C_{-\bar \theta}}(n_K(y))\! \,\dlat \mathcal{H}^{n-1}(y), \quad \bar\theta=(\theta_1,\dots,\theta_\hid)\in \S,
\end{equation}
where $C_{\bar x}\coloneqq\conv\{o,x_1,\dots,x_\hid\}\subset \R^n$ is a polytope whose vertices are a subset of $\{x_i\}_{i=1}^m$. The exact formula for its support function is
\[
h_{C_{-\bar \theta}}(u) = \max_{1\leq i \leq \hid}\langle u,\theta_i \rangle_{-},\quad \text{ where, for }a\in\R,\;  a_-=\max\{0,-a\}.
\]
This definition was motivated by the fact that, for every $\bar\theta\in \mathbb{S}^{nm-1},$ 
\begin{equation}
    \frac{\partial}{\partial r}\Cov(r\bar\theta)\Big|_{r=0^+}=-\|\bar\theta\|_{\PP K},
    \label{eq:cov_deriv}
\end{equation}
as established in \cite{HLPRY25}. Note that, for every convex body $K$ and $\theta\in\s$,
\begin{equation}
\label{eq:reduces}
\int_{\partial K} \langle \theta,n_K(y) \rangle_{-} \dlat \mathcal{H}^{n-1}(y) = \frac{1}{2}\int_{\partial K} |\langle \theta,n_K(y) \rangle| \dlat \mathcal{H}^{n-1}(y),\end{equation}
establishing that $\Pi^{\circ,1} K =\Pi^\circ K$. The following inequalities hold: for $m,n\in\N$ and $K\subset \R^n$ a convex body, 
\begin{equation}
\label{eq:zhang_higher}
\frac{1}{n^{n\hid}}\binom{n(\hid+1)}{n}\leq \frac{\vol_{n\hid}(\PP K)}{\vol_n(K)^{\hid(1-n)} }\leq \frac{\vol_{n\hid}(\PP \B)}{\vol_{n}(\B)^{m(1-n)}}.\end{equation}
The first inequality in \eqref{eq:zhang_higher} is the $m$th-order Zhang's projection inequality, with equality if and only if $K$ is a $n$-dimensional simplex (see \cite[Corollary 1.3]{HLPRY25}), and the second inequality is the $m$th-order Petty's projection inequality with equality if and only if $K$ is an ellipsoid (see \cite[Theorem 1.4]{HLPRY25}).

\subsection{The theory of log-concave functions.} Our next framework is the setting of $\log$-concave functions. In what follows, a measurable function is a Borel measurable function, if not differently specified. For a measurable function $g:\R^n\longrightarrow \R_+$, where $\R_+=[0,\infty)$, we denote as usual $$\|g\|_{p} = \left(\int_{\R^n}g^p\,\dlat x\right)^\frac{1}{p} \quad \text{and} \quad \|g\|_\infty = {\operatorname{ess} \sup} \;g.$$
The support of $g$ is the closed set $\supp(g) = \overline{\{x\in\R^n:g(x)>0\}}$.

A function $f:\R^n\longrightarrow \R_+$ is $\log$-concave if it is non-identically zero and for every $\lambda \in [0,1]$ and $x,y\in \R^n$, one has 
\begin{equation}f((1-\lambda)x + \lambda y) \geq f(x)^{1-\lambda}f(y)^\lambda.
\label{eq:log}
\end{equation}
In this work, we will always consider the following class of log-concave functions
\begin{align*}
\lc:&=\left\{f:\R^n\longrightarrow \R_+:f \not\equiv 0,\text{ is log-concave,} \right.
\\
&\quad\quad\quad\quad\left. \text{upper semi-continuous, and integrable}.\right\}
\end{align*}

The works of Ball \cite{Ball88}, Artstein-Avidan, Klartag, and Milman \cite{AAKM04}, and Klartag and Milman \cite{KM05} initiated the geometrification of $\log$-concave functions, where results and behaviours similar to those of convex bodies are shown to hold for $f\in \lc$. Several fundamental inequalities from convex geometry have received extensions to this setting, such as Santal\'o's inequality (and its reverse inequality) \cite{AAKM04,FM07,FM08,FM08_2,JL08,JL09,FGMR10,BBF14,FGZ23}, Gr\"unbaum's inequality \cite{MF24,MNRY18,MSZ18}, Petty's projection inequality \cite{GZ99,TW12,YL21,SAL19,HLPRY25,HLPRY23_2}, Zhang's projection inequality \cite{ABG20, AGJV18,LRZ22} and the Rogers-Shephard inequality \cite{Co06,AlGMJV,AHNRZ, AAEFO15,AAGJV}.

We write $\mathrm{BV}(\R^n)$ to denote the space of functions of bounded variation on $\R^n$. We mention here that given a function in $\mathrm{BV}(\R^n)$, there exists a vector-valued measure on $\R^n$ associated to it, its so-called \textit{variation measure}. One has that $\lc \subset \mathrm{BV}(\R^n)$. Recently, Rotem \cite[Proposition 2.5]{LR23} established a formula for the variation measure of $f\in \lc$; for a self-contained study of log-concave functions, we choose to use this explicit representation in our definitions. We recall that, for $f\in \lc$, $\nabla f$ exists almost everywhere on $\R^n$. Moreover, $\supp(f)$ is convex and, therefore, $\partial \supp(f)$ admits $\mathcal{H}^{n-1}$-almost everywhere an outer-unit normal. A fundamental instrument in our investigation is the LYZ body, originally introduced by Lutwak, Yang, and Zhang in \cite{LYZ06}. We present it, for now, in a self contained form, postponing to the next section a more detailed discussion on the topic.
\begin{definition}
\label{def:LYZ_body}
    Let $f\in \lc$. Then, its asymmetric LYZ body $\langle f \rangle\subset\R^n$ is the unique convex body with center of mass at the origin with the following property: for every $1$-homogeneous function $h$ on $\Rn\setminus\{o\}$,
\begin{equation}\int_{\partial \langle f \rangle}\!\!h(n_{\langle f \rangle}(y))\,\dlat \mathcal{H}^{n-1}(y)\!=\!\int_{\Rn}\!\!\!h(-\nabla f)\,\dlat x + \!\!\int_{\partial\supp(f)}\!\!\! \!\!\!\!\!\!\!\!h(n_{\supp(f)} (y)) f(y)\dlat\mathcal{H}^{n-1}(y).
\label{eq:LYZ}
\end{equation}
\end{definition}\noindent
We address the existence and uniqueness of this case of the asymmetric LYZ body in Section~\ref{sec:fun}. 

The Zhang's projection inequality from the right-hand side of \eqref{eq:Zhang_other_form} was extended to $ \lc$ by Alonso-Guti\'errez, Bernu\'es, and Gonz\'alez Merino \cite{ABG20}, where they had to work with the projection operator $\pp$ applied to the level sets of $f$. The notation $\Pi^\ast f$ is commonly used in this context, see, e.g., \cite{ABG20,AlGMJV,SAL19,AGJV18,HL25}.

For our purposes, it is possible, and indeed crucial, to state the functional polar projection body in terms of the asymmetric LYZ body from Definition \ref{def:LYZ_body}. The idea of applying the operator $\pp$ to the LYZ body appeared implicitly in a few works by G. Zhang \cite{GZ99} (for $C^1$ smooth functions), Lutwak, Yang and Zhang \cite{LYZ02,LYZ06,CLYZ09} (for $W^{1,1}(\R^n)$), and later explicitly in the works of T. Wang \cite{TW12,TW13} (for $\mathrm{BV}(\R^n)$) and \cite{CLM17_2} (for $\lc$).

 Hence, within our work, the polar projection body of a function $f\in\lc$ is the result of the operator $\Pi^\circ$ applied to the asymmetric LYZ body $\langle f \rangle$, i.e.,  $\pp \langle f \rangle\subset \R^n$. By Definition~\ref{def:LYZ_body}, with $h=\frac{1}{2}|\langle \theta, \cdot\rangle|$, we see that $\Pi^\circ \langle f \rangle$ is the unit ball of the norm 
\begin{equation}
    \|x\|_{\pp \langle f \rangle} \!\!= \!\frac{1}{2}\!\left(\!\int_{\Rn}\!\!\!|\langle \nabla f(x),\theta\rangle|\,\dlat x \!+\! \!\!\int_{\partial\supp(f)} \!\!\!\!\!\!\!\!\!\!\!\!\!\!\!\!|\langle n_{\supp(f)} (y), \theta \rangle| f(y)\dlat\mathcal{H}^{n-1}(y)\!\!\right).
    \label{eq:ppLYZ}
\end{equation}
 Comparing our definition to those aforementioned, the cited works consider either $f$ in the Sobolev space $W^{1,1}(\R^n)$, and so without the integral over $\partial \supp(f)$ found in \eqref{eq:ppLYZ}, or in $\mathrm{BV}(\R^n)$ with the abstract variation measure in place of our specific formulas from the more recent work \cite{LR23}. Under our normalization, we have $$\pp \langle e^{-\|x\|_K} \rangle = \frac{1}{(n-1)!} \pp K.$$

Written in terms of $\pp \langle f\rangle$, the main result \cite[Theorem 2.1.1]{ABG20} by Alonso-Guti\'errez, Bernu\'es, and Gonaz\'alez Merino is as follows: for $f\in \lc$, they introduced the functional covariogram
\begin{equation}
    \label{eq:fun_covario}
    g_f(x)=\int_{\R^n}\min\left\{f(y),f(y-x)\right\}\,\dlat y
\end{equation}
and proved
\begin{equation}\frac{1}{n!}\int_{\R^n}g_f(x)\,\dlat x\leq \|f\|_1^{n+1}\vol_n(\pp \langle f \rangle).
\label{eq:Zhang_fun}
\end{equation}
Among the functions with $\|f\|_\infty = f(o)$, there is equality if and only if $f(x)=\|f\|_\infty e^{-\|x\|_{\Delta_n}}$ for any $n$-dimensional simplex $\Delta_n$ containing the origin. We note that the constants in the original work are slightly different, again due to our definition of $\pp \langle f \rangle$. 

In our first main result, we prove the $m$th-order generalization of Zhang's projection inequality (i.e., the first inequality in \eqref{eq:zhang_higher}) for integrable $\log$-concave functions. Our primary object of interest is obtained by applying the operator $\PP$ to $\langle f \rangle$, creating the $m$th-order polar projection LYZ body of $f\in \lc$, $\PP \langle f\rangle\subset\R^{nm}$. Like in the $m=1$ case, we choose to utilize Definition~\ref{def:LYZ_body} in conjunction with \eqref{eq:PP}.
\begin{definition}
\label{def:higer_LYZ_body}
    Fix $m,n\in\N$. Let $f\in \lc$. Then, its $m$th-order polar projection LYZ body $\PP \langle f\rangle\subset\R^{nm}$ is the convex body containing the origin given by the gauge    
\begin{equation}
\label{eq:hi_LYZ_BV}
\begin{split}
\|\bar\theta\|_{\LYZP{f}} = \int_{\R^n} &h_{C_{-\bar \theta}}\left(-\nabla f(x)\right)\,\dlat x 
\\
&+ \int_{\partial\supp(f)} h_{C_{-\bar \theta}}(n_{\supp(f)} (y))f(y)\dlat\mathcal{H}^{n-1}(y).
\end{split}
\end{equation}
\end{definition}\noindent
The Definition~\ref{def:higer_LYZ_body} reduces to the formula \eqref{eq:ppLYZ} when $m=1$, thanks to \eqref{eq:reduces}. We establish the following result, extending the $m=1$ case from \cite{ABG20}. We henceforth set $x_0=o$. For $\bar x = (x_1,\dots,x_m)\in\R^{n\hid}$, we set the $m$th-order covariogram of a function to be 
\begin{equation}
    \label{eq:fun_covario_hi}
    g_{f,\hid}(\bar x)=\int_{\R^n}\min_{0\leq i \leq m}\{f(y-x_i)\}\,\dlat y.
\end{equation}
Observe also that $g_{f,\hid}(\bar o)=\|f\|_1$.
\begin{theorem}[The $m$th-order Zhang's projection inequality for log-concave functions]
 Fix $n,\hid\in\N.$ Let $f\in \lc$. Then, 
    \begin{equation}\frac{1}{(n\hid)!}\int_{\R^{n\hid}}g_{f,\hid}(\bar x)\,\dlat\bar x\leq \|f\|_1^{n\hid+1}\vol_{n\hid}(\PP \langle f \rangle).
    \label{eq:m_fun_Zhang}
\end{equation}
where equality holds if and only if $f(x)=\|f\|_\infty e^{-\|x-x^\prime\|_{\Delta_n}}$ for some $x^\prime\in\R^n$ and an $n$-dimensional simplex $\Delta_n$ containing the origin.
\label{t:Zhang_fun_hi}
\end{theorem}
Using the classical formula 
\begin{equation}\vol_n(K)=\frac{1}{n!}\int_{\R^n}e^{-\|x\|_K}\,\dlat x
\label{eq:classical}
\end{equation} and the fact that the maximum of gauges corresponds to the gauge of the Cartesian product, Theorem~\ref{t:Zhang_fun_hi} indeed reduces to the left-hand inequality in \eqref{eq:zhang} when $f=e^{-\|\cdot\|_K}$.

We believe that our proof of the inequality \eqref{eq:m_fun_Zhang} (even when $m=1$) provides a substantial simplification of the strategy adopted in \cite{ABG20} due to the usage of the asymmetric LYZ body. In Section~\ref{sec:fun_covario}, we provide a direct, concise proof of Theorem~\ref{t:Zhang_fun_hi}. Additionally, in Section~\ref{sec:radial_function_bodies}, we will realize Theorem~\ref{t:Zhang_fun_hi} as a special case of a more general theorem.

\subsection{Radial mean bodies.} We first introduce a larger geometric setting. In \cite{GZ98}, Gardner and Zhang introduced the radial $p$th mean bodies $R_p K \subset \R^n$ of a convex body $K\subset \R^n$, where $p>-1$. Given a \textit{star body} $L\subset\R^n$ (see Definition~\ref{def:star}), it is uniquely identified with its \textit{radial function} \begin{equation}
    \label{eq:radial_function}
    \rho_L(x)\coloneqq\sup\{t>0:tx\in L\}, \quad x\in \R^n\setminus\{o\}.\end{equation} We note that a convex body $K$ is always a star body if translated by a point in its interior. The radial $p$th mean bodies $R_pK$ of a convex body $K\subset \R^n$ from \cite{GZ98} are star bodies given by the radial functions
    \begin{equation}
    \rho_{R_p K}(\theta) = \begin{cases}\left(\frac{1}{\vol_n(K)}\int_K\rho_{K-x}(\theta)^p\,\dlat x\right)^\frac{1}{p}, & p>-1,p\neq 0,
    \\
    \exp\left(\frac{1}{\vol_n(K)}\int_K\log\left(\rho_{K-x}(\theta)\right)\,\dlat x\right),  & p=0,
    \\
    \max_{x\in K}\rho_{K-x}(\theta)=\rho_{DK}(\theta), & p=\infty.
    \end{cases}
    \label{eq:radial_mean_og}
\end{equation}
An immediate consequence of the definition is the following (strict) chain of set inclusions, which follows by a direct application of Jensen's inequality to \eqref{eq:radial_mean_og}: for $-1 < p < q < \infty,$
\begin{equation}
\label{eq:chain_og}
    \{o\} = R_{-1}K \subset R_p K \subset R_q K \subset R_{\infty }K = DK.
\end{equation}
More importantly, a consequence of \cite[Theorem 2.2]{GZ98}, is the following limit
\begin{equation}
    \lim_{p\to (-1)^+}(1+p)^\frac{1}{p}R_p K = \vol_n(K)\Pi^\circ K.
\end{equation}
Additionally, \cite[Theorem 5.5]{GZ98} establishes a formal reversal of \eqref{eq:chain_og}: for $-1 < p < q < \infty$, one has the chain of inequalities
\begin{equation}
    \label{e:set_inclusion}
    R_\infty K = D K \subseteq \binom{n+q}{q}^\frac{1}{q} R_{q} K \subseteq \binom{n+p}{p}^\frac{1}{p} R_{p} K \subseteq n \vol_n(K) \Pi^{\circ} K.
\end{equation}
 Equality in \eqref{e:set_inclusion} is achieved in each inclusion if and only if $K$ is an $n$-dimensional simplex. The additional observation that $$\vol_n(R_n K)=\vol_n(K)$$ shows that \eqref{e:set_inclusion} interpolates between \eqref{eq:RS} and \eqref{eq:zhang}. 

Through a suitable change of variable, one infers the following equivalent form of \eqref{eq:radial_mean_og} for the radial function of $R_p K$. The identity is well-known (see \cite[Lemma 3.1]{GZ98} and also \cite{LP25} for $p\in (-1,0)$), but we will provide a proof for completeness. Recall that $g_K$ is the covariogram of $K$ from \eqref{eq:covariogram}.
\begin{proposition}
\label{p:equivalence}
Let $p>-1$ and $K\subset \R^n$ a convex body. Then, the radial function of $R_p K$ satisfies, for $\theta\in\s$:
    \begin{equation}
\rho_{R_p K}(\theta)=\begin{cases}
    \left(p\int_0^{\infty}\left(\frac{g_K(r\theta)}{\vol_n(K)}\right)r^{p-1}\dlat \, r\right)^\frac{1}{p}, & p>0,
    \\
    \exp\left(\int_0^{\infty}\frac{\partial}{\partial r}\left(\frac{-g_K(r\theta)}{\vol_n(K)}\right)\log(r)\dlat \, r\right), & p=0,
    \\
    \left(p\int_{0}^{\infty} \left(\frac{g_K(r\theta)}{\vol_n(K)}-1\right) r^{p-1} d r\right)^\frac{1}{p}, & p\in (-1,0).
\end{cases}
\label{radial_ell}
\end{equation}
In particular, the sets $R_p K$ are origin-symmetric star bodies for all $p >-1$ and convex for $p\in [0,\infty)$.
\end{proposition}
For $p\in (-1,0)$, the convexity of $R_p K$ is an open problem; it was recently established in the special case $n=2$ \cite{JH25}. The formula for $p>0$ in \eqref{radial_ell} shows that $R_p K$ are the so-called $p$th \textit{Ball bodies} of $g_K$, introduced by K. Ball in the influential work \cite{Ball88}. Motivated in part by the formula \eqref{radial_ell}, we introduce Ball bodies for the choice $p\in (-1,0]$, which, to the extent of our knowledge, appears to be missing from the literature. For a precise definition of Ball bodies, see Section~\ref{sec:ball_bodies}. The idea of this construction is to consider a function $g:\R^n\longrightarrow \R_+$ with suitable integrability assumptions and produce a star body $K_p (g)\subset \R^n$, $p>-1$. In this spirit, the identity \eqref{radial_ell} can be expressed as
\[R_p K = K_p\left(g_K\right).\]
The $m$th-order extensions $R^\hid_p K\subset \R^{nm}$ of $R_p K$ were defined in \cite{HLPRY25} for $m\geq 1$; under the present notation, they are
 \[
 R_p^mK=K_p(g_{K,m}),
 \]
 and they satisfy similar inclusions to those of \eqref{eq:chain_og} and \eqref{e:set_inclusion}. 
 
A further subject of this work is a functional generalization of the radial mean bodies $R^\hid_p K$ and the corresponding version of \eqref{e:set_inclusion} for log-concave functions. Even in the case $m=1$, these bodies, and the corresponding results, appear to be new. We denote these bodies by $R^\hid_p f$. See Section~\ref{sec:radial_function_bodies} for the precise definition, which can be synthetically expressed as
 \[
R_p^m f =K_p(g_{f,m}) \subset \R^{nm}.
\]
In analogy with radial pth mean bodies of convex bodies, one has for $p\geq 0$ that $R^\hid_p f$ is a convex body due to results from \cite{Ball88,GZ98} (see Proposition~\ref{p:radial_ball}). If $\hid=1$, then $R_p^m f$ is always origin symmetric since $g_{f}$ is always an even function. If $\hid \geq 2$, $R^\hid_p f$ may not be origin-symmetric, as is the case when $f=\chi_K$ for a non-symmetric $K$ (see \cite[Proposition 3.5]{HLPRY25}). For $p\in (-1,0)$, $R^\hid_p f$, the question of convexity remains open.

For our purposes, we also study the limiting bodies as $p\to -1$ and $p\to \infty$, which are deduced from the analysis of the asymptotic behaviour of the more general Keith Ball bodies in Section~\ref{sec:ball_bodies}. As a result, we will show that \[ \lim_{p\to (-1)^+}(1+p)^\frac{1}{p}R^\hid_p f=\|f\|_1\PP \langle f \rangle. \]

Concerning the asymptotics for $p\to \infty$, we will see that $R_{\infty}^\hid f$ may be a compact, convex set, but it may also be $\R^{nm}$. Additionally, the limit $\lim_{p\to \infty}\Gamma\left(1+p\right)^{-\frac{1}{p}}R^\hid_p f$ will always be a compact, convex set, but it may be a singleton containing the origin, as the following statement shows. Recall the digamma function $\psi(t)=\Gamma^\prime(t)/\Gamma(t)$.
 \begin{corollary}
\label{cor:zero}
    Fix $m,n\in\N$. Let $K\subset \R^n$ be a convex body and set $\gamma_K(x)=e^{-\frac{\|x\|_K^2}{2}}$. Then, $R_0^m\gamma_K = \sqrt{2}\exp\left(\frac{1}{2}\psi\left(1+\frac{n}{2}\right)\right)R_0^m K$ and, for $p>0$,
\[
R^\hid_p \gamma_K = \sqrt{2}\left(\frac{\Gamma\left(1+\frac{n+p}{2}\right)}{\Gamma\left(1+\frac{n}{2}\right)}\right)^\frac{1}{p}R_p^{\hid}K.
\]
In particular,
\[
    \lim_{p\to \infty}\Gamma(1+p)^{-\frac{1}{p}}R^\hid_p \gamma_K =\{o\}.
    \]
\end{corollary}

Due to the above considerations, we choose not to treat the case $p=\infty$ in the following theorem, as it may be a singleton containing the origin.
\begin{theorem}
\label{t:radial_hi_fun_set}
    Fix $n,m\in\N$. Let $f\in\lc$. Then, for $-1<p<q<\infty$,
    $$ \left(\frac{1}{\Gamma(q+1)}\right)^\frac{1}{q} R_{q}^\hid f \subseteq \left(\frac{1}{\Gamma(p+1)}\right)^\frac{1}{p} R_{p}^\hid f \subseteq \|f\|_1\PP \langle f \rangle,$$
    with equality if and only if 
    $$g_{f,\hid}(\bar x) = \|f\|_1e^{-\frac{\|\bar x\|_{\PP \langle f \rangle}}{\|f\|_1}}.$$
    Equivalently, there is equality if and only if  $f(x)=\|f\|_\infty e^{-\|x-x^\prime\|_{\Delta_n}}$ for an $n$-dimensional simplex $\Delta_n$ containing the origin and some $x^\prime \in \R^n$.
\end{theorem}

\subsection{Functional, higher-order Rogers-Shephard inequality} In the classical setting, as mentioned, $DK = \text{supp}(g_K)$, and the difference body appears in the set inclusions \eqref{eq:chain_og} and \eqref{e:set_inclusion}. We have a similar characterization.
\begin{proposition}
\label{p:support}
    Fix $n,m\in\N$. Let $f\in \lc$. Then, $\supp(g_{f,\hid})\subseteq \R^{nm}$ contains the origin in its interior and
        \[
        R_\infty^\hid f = \lim_{p\to\infty} R_p^\hid f = \supp(g_{f,\hid}) = D^m(\supp(f)).
        \]
If $\supp(f)=\R^n$, the above should be understood as $\R^{nm}$.
\end{proposition}\noindent
An inequality involving $\supp(g_{f,m})$ seems therefore an unnatural substitute for the Rogers-Shephard inequality, especially in the case when $\supp(f)=\R^n$; indeed, an inequality involving $D^m(\supp(f))$ would erase too much of the information defining $f$. In these terms, in stark contrast with the geometrical case, a functional Rogers-Shepard-type inequality cannot be deduced from Theorem \ref{t:radial_hi_fun_set}. Therefore, in order to complete the picture and establish such inequalities, alternative choices of ``covariograms" are required.

The first step is to introduce an appropriate generalization of Minkowski addition. For two log-concave functions $f,g:\R^n \longrightarrow \R_+$, their sup-convolution is defined as $$f\star g(z)=\sup_{x+y=z}f(x)g(y).$$
Setting $\overline f(x)=f(-x)$, Alonso-Guti\'errez, Gonz\'alez Merino, Jim\'enez, and Villa \cite[Theorem 2.2]{AlGMJV} showed, elaborating on a previous inequality by Colesanti \cite{Co06}, that
$$\int_{\R^n}f\star \bar f(z) \,\dlat z\leq \binom{2n}{n}\|f\|_\infty\|f\|_1,$$
with equality if and only if $f/\|f\|_\infty$ is the characteristic function of an $n$-dimensional simplex.
In fact, they also established (see \cite[Theorem 2.1]{AlGMJV}) a generalization when $\overline f$ is replaced by any log-concave function $g$.

Our final result is the following $m$th-order version of the Rogers-Shephard inequality \eqref{eq:RSell}, which extends the $\hid=1$ case established by Alonso-Guti\'errez, Gonz\'alez Merino, Jim\'enez, and Villa \cite{AlGMJV}. Given $\hid+1$ functions $f_0,\!\dots\!,f_\hid$ on $\R^{n}$, define $\bar f((x_0,\bar x)) \in \R^{n(\hid+1)}$ as \[\bar f((x_0,\bar x))\! =\! (f_0(x_0),f_1(x_1),\!\dots\!,f_\hid(x_\hid)).\] We next consider the convolutions
$$
(\bar f)_{\oplus_{\hid}} (\bar x) \coloneqq \int_{\R^n} f_0(z)f_1(z - x_1) \cdot \cdot \cdot f_\hid(z - x_\hid) \dlat z
$$
and
$$
(\bar f)_{\star_{\hid}} (\bar{x}) \coloneqq \sup_{z \in \R^n} f_0(z)f_1(z - x_1) \cdot \cdot \cdot f_\hid(z - x_\hid).
$$
With these definitions at hand, we follow the ideas from \cite{AlGMJV} and prove the following Rogers-Shephard-type theorems. The first applies to the integral convolution of $\hid+1$ log-concave functions.
\begin{theorem}
\label{t:hi_fun_RS}
    Fix $n,\hid\in\N$. Let $f_0,f_1,...,f_m \in \lc$. Then, setting $\bar f = (f_0,f_1,\dots,f_\hid):$
    \begin{align*}
        \|(\bar f)_{\oplus_{\hid}}\|_\infty\| (\bar f)_{\star_{\hid}}\|_{1}
        \leq \binom{n(\hid + 1)}{n} \left(\prod_{i=0}^\hid\|f_i\|_\infty\|f_i\|_1\right).
    \end{align*}
\end{theorem}\noindent
When setting $f_0=f$ and replacing, for $i=1,\dots,\hid$, each $f_i$ with $f(-\cdot)$ in Theorem~\ref{t:hi_fun_RS}, we can obtain a sharper Rogers-Shephard-type inequality using a slightly different approach.
\begin{theorem}
\label{t:hi_fun_RS_one}
    Fix $n,\hid\in\N$. Let $f\in \lc$ and set $\bar{f}\! \coloneqq\!\! (\!f(\cdot),\!f(-\cdot),\!\cdots\!,\!f(-\cdot)\!)$. Then:
    \begin{equation}\label{e:hi_fun_RS_one}
       \|(\bar{f})_{\star_{\hid}}\|_1 \leq \binom{n(\hid + 1)}{n} \|f\|_\infty^m \|f\|_\frac{1}{m},
    \end{equation}
    where equality holds if and only if $f/\|f\|_\infty$ is the characteristic function of an $n$-dimensional simplex.
\end{theorem}

\medskip

The outline of this work is the following. In Section~\ref{sec:prelim}, we recall some basic facts from convex geometry and the theory of log-concave functions. In Section~\ref{sec:class}, we introduce Ball bodies, discuss their history and extend them to $p \in (-1,0)$. We study in Section~\ref{sec:ball_bodies} their asymptotics. Section~\ref{sec:hi} is dedicated to proving Theorem~\ref{t:Zhang_fun_hi}, Corollary~\ref{cor:zero}, and Theorem~\ref{t:radial_hi_fun_set}. Finally, Section~\ref{sec:rs} is devoted to the proofs of Theorem~\ref{t:hi_fun_RS} and Theorem~\ref{t:hi_fun_RS_one}. 

\section{Preliminaries}
\label{sec:prelim}
One of the cornerstones of the theory of convex bodies is the Brunn-Minkowski inequality. This result states that if $K$ and $L$ are convex bodies and $\lambda\in[0,1]$, then
\begin{equation}\label{eq:BM}
    \vol_n((1-\lambda) K + \lambda L) \geq \vol_n(K)^{1-\lambda}\vol_n(L)^\lambda.
\end{equation} 
Using the homogeneity of the volume, one may easily check that \eqref{eq:BM} is equivalent to 
\begin{equation}\label{e:BM}
    \vol_n((1-\lambda) K + \lambda L)^\frac{1}{n} \geq (1-\lambda)\vol_n(K)^\frac{1}{n}+\lambda \vol_n(L)^\frac{1}{n},
\end{equation}
where equality holds for some $\lambda \in (0,1)$ if and only if $K$ and $L$ are homothetic (i.e., $K=tL +x$, for some $t \in \R_+, x \in \R^n$).

 The functional counterpart of the Brunn-Minkowski inequality \eqref{eq:BM} is the Pr\'ekopa-Leindler inequality (\cite{PreL1, PreL2}, see also \cite{Lei72b, BL76}). In its general statement, it reads as follows: if $\lambda \in (0,1)$ and $h,f,g : \R^n \longrightarrow \R_{+}$ are non-negative measurable functions such that, for any $x,y \in \R^n$, $$h((1-\lambda)x+\lambda y)\geq f(x)^{1-\lambda}g(y)^\lambda,$$
then
\begin{equation}\int_{\R^n}h(z)\,\dlat z \geq \left(\int_{\R^n} f(x) \,\dlat x\right)^{1-\lambda}\left(\int_{\R^n} g(y) \,\dlat y\right)^{\lambda}.
\label{eq:PL}
\end{equation}
In particular, if $f,g \in \lc$, \eqref{eq:PL} holds with $h=f \star g$. The equality cases are treated in the work of Dubuc \cite{SD77}. We refer to Gardner's survey \cite{G20} for more details on the Brunn-Minkowski inequality and its connections to the Pr\'ekopa-Leindler inequality. 

We will also make use of the following fact concerning integrable, log-concave functions; see, e.g. \cite{BK07}.
\begin{proposition}
\label{p:integrable}
    Let $f:\R^n\longrightarrow \R_+$ be a non-identically zero, log-concave function. The following properties are equivalent:
    \begin{enumerate}
        \item[(i)] Integrability: $f \in \lc$,
        \item[(ii)] Coercivity: there exists some constants $A,B>0$ such  that for all $x\in \R^n$, 
\begin{equation}\label{eq:majo_log_concave}
f(x)\le A\, e^{-B\, |x|}.
\end{equation}
    \end{enumerate}
    In particular, $f$ has finite moments of all orders.
\end{proposition}

\subsection{The LYZ body of log-concave functions}
\label{sec:fun}
The notion of LYZ body was introduced by Lutwak, Yang, and Zhang \cite{LYZ06} to the Sobolev space $W^{1,1}(\R^n)$. Later, it was extended by Wang \cite{TW12} to the family of functions of bounded variation $\mathrm{BV}(\R^n)$. This construction appears implicitly in the seminal work by Zhang \cite{GZ99} on the affine Sobolev inequality. In our exposition, we are interested in this construction restricted to the space $\lc$, which is a subset of $\mathrm{BV}(\R^n)$ (compare \cite{LR23} for an insightful treatment of this inclusion). As the case of $\lc$ allows several simplifications, we propose, for the convenience of the reader, a streamlined proof of the construction of the LYZ body for integrable, log-concave functions. 

We need the following version of the co-area formula for log-concave functions, which is a consequence of \cite[Proposition 2.4 and Proposition 2.6]{LR23}. 
\begin{lemma} \label{lemma:coarea_rotem}
    Let $f\in \lc$ and let $h:\R^n\longrightarrow\R_+$ be a real-valued, $1$-homogeneous function on $\R^{n}\setminus\{0\}$. Then, 
    \begin{equation}
    \label{eq:coarea}
    \begin{split}
        \int_{\Rn}h(-\nabla f)\,\dlat x &+ \!\!\int_{\partial\supp(f)} h(n_{\supp(f)} (y)) f(y)\dlat\mathcal{H}^{n-1}(y)
        \\
        &= \int_0^{\|f\|_\infty}\int_{\partial \{f\geq t\}} h(n_{\{f \geq t\}}(y))\,\dlat\mathcal{H}^{n-1}(y)\,\dlat t.
    \end{split}
    \end{equation}
\end{lemma}\noindent
We can now justify Definition \ref{def:LYZ_body}. We recall that, for a convex body $K$, its area measure $S_K$ is the pushforward of $\mathcal{H}^{n-1}$ from $K$ to $\s$ by $n_{K}$.
\begin{proposition}\label{prop:existence_LYZ}
    Given $f \in \lc$, there exists a unique Borel measure $\mu_f$ on $\s$ which is centered and not concentrated on any great subsphere, such that  
    \begin{equation}
    \begin{split}
    \int_{\s}&h(u)\,\dlat{\mu_f}(u)
    \\
    &=\int_{\Rn}h(-\nabla f)\,\dlat x + \int_{\partial\supp(f)} \!\!\!\!\!h(n_{\supp(f)} (y)) f(y)\dlat\mathcal{H}^{n-1}(y)
    \end{split}
    \label{eq:change_of_variable}
    \end{equation}
    for every $1$-homogeneous function $h$ on $\Rn\setminus\{o\}$.

    In particular, there exists a unique convex body $\langle f \rangle$ with center of mass at the origin and such that $\mu_f=S_{\langle f \rangle}$.
\end{proposition}
\begin{proof}
    By \eqref{eq:change_of_variable}, the definition of area measure, we infer 
    \begin{align*}
        & \int_{\Rn}\!\!\!h(-\nabla f)\,\dlat x + \!\!\int_{\partial\supp(f)} \!\!\!\!\!\!\!\!h(n_{\supp(f)} (y)) f(y)\dlat\mathcal{H}^{n-1}(y)\\
        =& \int_0^{\|f\|_\infty}\int_{\partial \{f\geq t\}} h(n_{\{f \geq t\}}(y))\,\dlat\mathcal{H}^{n-1}(y)\,\dlat t\\
        =& \int_0^{\|f\|_\infty}\int_{\s} h(u) \dlat S_{\{ f \geq t\}}(u) \dlat t.
    \end{align*}
    The last line provides a positive linear functional on $\s$, which is bounded thanks to Proposition \ref{p:integrable} (in particular, the coercivity of $f$ implies that the level sets $\{f \geq t\}$ are compact for every $t \geq 0$). Therefore, by the Riesz representation theorem, there exists a unique Radon measure $\mu_f$ on $\s$ such that \[  \int_{\s}h(u)\,\dlat{\mu_f}(u)=\int_0^{\|f\|_\infty}\int_{\s} h(u) \dlat S_{\{ f \geq t\}}(u) \dlat t,\] thus proving \eqref{eq:change_of_variable}. The fact that $\mu_f$ is centered follows from the fact that each $S_{\{f \geq s\}}$ is centered (compare \cite[Section 8]{Sh1}). The fact that $\mu_f$ is not concentrated on any great subsphere follows from the choice $h(x)=|\langle x, u \rangle|, u \in \s$, as proved in \cite[Proposition 3.1]{FTLR25}.

    Finally, the existence of $\langle f \rangle$, up to translation, follows from the classical Minkowski problem.  By choosing $\langle f \rangle$ to have center of mass at the origin, it becomes unique. See, for example, \cite[Section 8.2]{Sh1}.
\end{proof}

In usual applications in the literature (compare \cite{ML12,TW12,LUD,CLM17_2}), the LYZ body of $f$ is the unique origin symmetric convex body whose surface area measure equals the \textit{even} part of $\mu_f$. The symmetric LYZ body is a fundamental tool in establishing affine-Sobolev-type inequalities \cite{GZ99,LYZ06,HJM16,TW12,LYZ02,HL25}. Similar applications of the asymmetric body $\langle f \rangle$ can be found in \cite{HSX12,HLPRY25}. We emphasize that, when dealing with $f \in \lc$, the body $\langle f \rangle$ is not necessarily symmetric, since the level sets of $f$ are always convex and bounded and, therefore, the measure $\mu_f$ we retrieve in Proposition \ref{prop:existence_LYZ} is centered by construction, without requiring any further step.

Finally, we remark that the work of Rotem \cite{LR23} on the total variation measure of log-concave functions was a culmination of a series of works \cite{CF13, CEK15, LR22, LR23}, showing in particular that it appears naturally as the first variation of the integral of suitable perturbations of log-concave functions. See \cite{NDZY25,FTLR25,HLXZ24,UJ25} for further developments.

\section{Sets associated with functions}
\label{sec:class}
This section is dedicated to obtaining \eqref{eq:chain_og}, \eqref{e:set_inclusion}, and Theorem~\ref{t:radial_hi_fun_set} as special cases of a more general statement. We begin by recalling that a function $f:\R^n\longrightarrow\R_+$ is $s$-concave, $s>0$, if it is upper semi-continuous, not identically zero, and, for $\lambda \in [0,1]$ and $x,y\in\R^n$ such that $f(x)f(y)>0$, one has
 \begin{equation}f((1-\lambda)x + \lambda y) \geq \left((1-\lambda)f(x)^s+\lambda f(y)^s\right)^\frac{1}{s}.
\label{eq:s_concave}
\end{equation}
As $s\to 0^+$, one recovers the log-concavity \eqref{eq:log}. By Jensen's inequality, every $s$-concave function is also $\log$-concave on its support. In particular, the implication that integrability implies coercivity in Proposition~\ref{p:integrable} still holds when $s>0$. Indeed, if $f$ is integrable and $s$-concave with $s>0$, then $\|f\|_\infty<\infty$ and $f \leq \|f\|_\infty\cdot \chi_{\supp(f)}$; we conclude by applying Proposition~\ref{p:integrable} to $\chi_{\supp(f)}$, which is log-concave since $\supp(f)$ is a compact, convex set. For ease of presentation, we may refer to $\log$-concavity as $s$-concavity, with $s=0$, in some of our results in this section.

To phrase these results, we start by recalling the Mellin transform. Although it is a classical operator from complex analysis, we side-step the associated regularity considerations by discussing only its application to $s$-concave functions, with $s\geq 0$ fixed.

 \subsection{The Mellin transform of functions in one variable.} Let $\psi:\R_+\longrightarrow\R_+$ be an integrable, $s$-concave function, $s\geq 0$. Then, its Mellin transform is the analytic function
\begin{equation}
   \Mel{\psi}{p}= \begin{cases}
    \int_0^\infty t^{p-1}(\psi(t)-\psi(0))\,\dlat t & \text{for } p\in (-1,0),
    \\
    \int_0^\infty t^{p-1}\psi(t)\,\dlat t & \text{for } p>0.
    \end{cases}
    \label{eq:Mel}
\end{equation}
We mention in passing that an integrable, log-concave (e.g. $s$-concave) function has finite moments for all $p$ by Proposition~\ref{p:integrable}, thus justifying no restriction for when $p$ is positive. The fact that $\Mel{\psi}{p}$ is finite for $p\in (-1,0)$ (under the given assumptions on $\psi$) follows by a standard argument involving analytic continuation; see \cite{FLM20}.

Since $\psi$ is $s$-concave, it is right-differentiable at $0$. Thus, one can use the following formula, via integration by parts, for all $p>-1$, $p\neq 0$:
\begin{equation}
\label{eq:alt_mellin}
\Mel{\psi}{p} = \frac{1}{p}\int_0^{\infty}(-\psi^\prime(t))t^p\;\dlat t.
\end{equation}
This shows that the Mellin transform has a simple pole at $p=0$. We will never really consider $\Mel{\psi}{0}$, but we do use frequently the limit
\[
\lim_{p\to0^+}\left(p\Mel{\tilde \psi}{p}\right)^\frac{1}{p} = \exp\left(\int_0^{\infty}\left(-\frac{\psi^\prime(t)}{\|\psi\|_\infty}\right)\log(t) \dlat t\right),
\]
with the notation
\[
\tilde\psi = \frac{\psi}{\|\psi\|_\infty}.
\]

We will need the following result, which encodes some monotonicity properties of the Mellin transform. It was established by Milman and Pajor \cite[Lemma 2.1]{MP89} when $p>0$ (see also \cite[Lemma 2.2.4]{BGVV14}). For our aims, we further extend it to $p \in (-1,0]$. The proof is postponed to the appendix.
\begin{proposition}
\label{p:increasing_mono}
    Let $\psi:\R_+\longrightarrow\R_+$ be a bounded, measurable function that is right-differentiable at the origin. Define
    $$I_p(\psi) \!=\! \left(p\Mel{\tilde \psi}{p}\right)^\frac{1}{p} \!\!=\!\!\begin{cases} \left(\frac{p}{\|\psi\|_\infty}\int_0^\infty \psi(r)r^{p-1}\, \dlat r\right)^\frac{1}{p}, &p>0,
    \\
    \exp\left(\frac{1}{\|\psi\|_\infty}\int_0^\infty(-\psi(r))^\prime\log(r)\; \dlat r\right), &p=0,
    \\
    \!\left(\frac{p}{\|\psi\|_\infty}\int_0^\infty r^{p-1}(\psi(r)\!-\!\psi(0))\; \dlat r\right)^\frac{1}{p}, & p\in (-1,0).
    \end{cases}
    $$
    Then, $p\mapsto (I_p(\psi))^\frac{1}{p}$ is increasing for $p>0$, and is constant if and only if $\psi(r)=\|\psi\|_\infty\chi_{[0,a]}(r)$ almost everywhere for some $a>0$. 
    
    Additionally, if $\psi$ obtains its maximum at the origin and decays almost surely to zero at infinity, then $p\mapsto (I_p(\psi))^\frac{1}{p}$ is increasing for $p\in (-1,0)$, and is constant if and only if $\psi(r)=\|\psi\|_\infty\left(1-\chi_{[0,a]}(1/r)\right)$ almost everywhere for some constant $a>0$.
\end{proposition}

Consider the generalized Binomial coefficients
\begin{equation}
\label{eq:binomial_coefficients}
\binom{p+\frac{1}{s}}{p}=\left(p\cdot B\left(p,\frac{1}{s}+1\right)\right)^{-1} = \begin{cases}\frac{\Gamma\left(p+\frac{1}{s}+1\right)}{\Gamma\left(p+1\right)\Gamma\left(\frac{1}{s}+1\right)}, & s>0,
\\
\Gamma(p+1)^{-1}, & s=0.
\end{cases}
\end{equation}
Letting $\Psi$ be the digamma function and $\gamma$ the Euler-Mascheroni constant, we have 
\[
\lim_{p\to0^+}\binom{p+\frac{1}{s}}{p}^\frac{1}{p} = \begin{cases}
    \exp\left(\Psi\left(s+1\right)+\gamma\right), & s>0,
    \\
    e^\gamma,& s=0.
\end{cases}
\]
In \cite{KPY08}, Koldobsky, Pajor, and Yaskin proved the following result for log-concave functions, providing the reverse direction of Proposition~\ref{p:increasing_mono}. We state it more generally for $s$-concave functions, as established by Fradelizi, Li, and Madiman \cite{FLM20}.  
\begin{proposition}[The Mellin-Berwald inequality]
\label{p:mellin_berwald}
    Fix $s\geq 0$. Let $\psi : \R_+ \longrightarrow \R_+$ be an integrable, $s$-concave function obtaining its maximum at the origin. Define the function
\begin{equation}
    G_{\psi}(p)\!=\!\!\begin{cases}\left(\binom{p+\frac{1}{s}}{p}\left(p\Mel{\tilde \psi}{p}\right)\right)^\frac{1}{p}, &  p>-1,p\neq 0,
    \\
    \left(\!\lim_{\!p\to0^+}\!\!\!\binom{p+\frac{1}{s}}{p}^\frac{1}{p}\!\!\right)\!\exp\left(\!\int_0^{\infty}\!\!\left(\!\frac{-\psi^\prime(t)}{\|\psi\|_\infty}\!\right)\!\log(t) \!\dlat t\right)\!, & p=0.
    \end{cases}
    \label{eq:milman_psi}
\end{equation}
Then, $G_{\psi}(p)$ is non-increasing in $p$. It is a constant, say $G_{\psi}(p)\equiv \alpha>0$, if and only if $$\psi(t)=\begin{cases}\psi(0)\left(1-\frac{t}{\alpha}\right)_+^\frac{1}{s}, & s>0,
\\
\psi(0)e^{-\frac{t}{\alpha}}\chi_{(0,\infty)}(t), &s=0.
\end{cases}
$$
\label{p:mb}
\end{proposition}
\noindent Notice that the characterization of equality in Proposition~\ref{p:mb} is implied by the proof and is not explicitly stated; see \cite{LP25} for the details.

A natural question is the behaviour as $p \to \infty$ of $G_{\psi}(p)$. In the next proposition, we consider the case when $\psi$ has bounded support. We state it more generally for measurable functions, thus replacing the notion of support with that of essential support. Recall that the essential support of a measurable function $\psi:\R_+\longrightarrow \R_+$ is the closed set defined by 
$$\operatorname{ess}\supp(\psi)=\R_+ \setminus \Big\{t\in \R_+ \; : \; \exists r>0, \ 
\int_{t-r}^{t+r} \psi(t)\,\dlat t=0 \, \Big\}.$$
\begin{proposition}
    Let $\psi:\R_+\longrightarrow \R_+$ be a bounded, measurable function with essential support $[0,R]$ for some $R>0$ such that $\|\psi\|_\infty=1$. Then, \begin{equation}
    \lim_{p\to\infty}\left(p\Mel{\psi}{p}\right)^\frac{1}{p} = R.
    \label{eq:mellin_limit}
\end{equation}
\label{p:mellin_finite_supp}
\end{proposition}
The proof is again postponed until the appendix. The following corollary is an immediate consequence of Proposition~\ref{p:mellin_finite_supp}, the definition of $G_{\psi}(p)$ from \eqref{eq:milman_psi}, and the fact that
\begin{equation}
\label{eq:bimon_infty}
\lim_{p\to \infty}\binom{p+\frac{1}{s}}{p}^\frac{1}{p} = \chi_{(0,\infty)}(s)=\begin{cases}
    1, & s>0,
    \\
    0,& s=0.
\end{cases}
\end{equation} 
\begin{corollary}
 Let $\psi:\R_+\longrightarrow \R_+$ be an $s$-concave, $s\geq 0$, function obtaining its maximum at the origin that is supported on $[0,R]$ for some $R>0$. Then,
     \begin{equation}
\label{eq:min_G}
\min G_{\psi}(p)=\lim_{p\to \infty}G_\psi(p)= R\chi_{(0,\infty)}(s).
\end{equation}
\end{corollary}

In the next proposition, we consider the case when the support of $\psi$ is all of $\R_+$. Among $s$-concave functions, only log-concave functions can satisfy this assumption.
\begin{proposition}
\label{p:p_to_infty}
    Let $\psi\in \operatorname{LC}_1$ be such that it obtains its maximum at the origin and $\supp(\psi)=\R_+$. Then, 
    \begin{equation}
    \label{eq:to_infty_and_beyond}
    \lim_{p\to \infty} \left(p\Mel{\tilde \psi}{p}\right)^\frac{1}{p} =\infty.
    \end{equation}
    Additionally, there exists $C>0$ such that, for all $p>0$,
    \begin{equation}
    \label{eq:function_bounded}
         \left(p\Mel{\tilde \psi}{p}\right)^\frac{1}{p}\leq C\cdot\Gamma\left(1+p\right)^\frac{1}{p},
    \end{equation} with equality if and only if $\psi$ is an exponential. Therefore,
    \begin{equation}
    \label{eq:limit_bounded}
    \lim_{p\to\infty}G_\psi(p)=\lim_{p\to\infty}\left(\frac{p\Mel{\tilde \psi}{p}}{\Gamma\left(p+1\right)}\right)^\frac{1}{p}\leq C.
    \end{equation}
    Moreover, the function $\psi(t)=e^{-\frac{t^2}{2}}$ satisfies 
    \begin{equation}
    \label{eq:limit_zero}
      \lim_{p\to \infty}G_\psi(p)=\lim_{p\to\infty}\left(\frac{p\Mel{\tilde \psi}{p}}{\Gamma\left(p+1\right)}\right)^\frac{1}{p}=0.
    \end{equation}
\end{proposition}
\begin{proof}
Begin by defining, for $k \in \N$, $\psi_k:=\frac{\psi}{\|\psi\|}\cdot \chi_{[0,k)}$. Notice that $\psi_k$ has maximum at the origin, and, in particular, $\|\psi_k\|_{\infty}=1.$ Since $\psi/\|\psi\|_\infty$ dominates $\psi_k$, and $\psi$ is integrable, we may use the dominated convergence theorem and obtain that, for every $p>0$, 
\[
\left(p\Mel{\tilde \psi}{p}\right)^\frac{1}{p} = \lim_{k\to\infty}\left(p\Mel{\psi_k}{p}\right)^\frac{1}{p}.
\]
Taking the limit in $p$, we have by Proposition~\ref{p:mellin_finite_supp}
\begin{align*}
\lim_{p\to\infty}\left(p\Mel{\tilde \psi}{p}\right)^\frac{1}{p} &=\lim_{p\to\infty} \lim_{k\to\infty}\left(p\Mel{\psi_k}{p}\right)^\frac{1}{p}
\\
&= \lim_{k\to\infty}\lim_{p\to\infty}\left(p\Mel{\psi_k}{p}\right)^\frac{1}{p}
= \lim_{k\to\infty}k=\infty,
\end{align*}
where we used the monotone convergence theorem (of sequences) to interchange the limits in $p$ and $k$. This establishes \eqref{eq:to_infty_and_beyond}.

    As for \eqref{eq:function_bounded} and \eqref{eq:limit_bounded}, we have by the assumptions on $\psi$ and Proposition~\ref{p:integrable}, that there exists $C>0$ such that
    \[
    \psi(t) \leq \|\psi\|_\infty  e^{-\frac{t}{C}}, \quad \forall\; t>0.
    \]
    Consequently, for every $p>0$
    \begin{align*}
    \frac{p}{\|\psi\|_\infty}\int_0^{\infty}\psi(t)t^{p-1} \; \dlat t 
    \leq p\int_0^{\infty}e^{-\frac{t}{C}}t^{p-1}\; \dlat t = C^p\cdot\Gamma(p+1),
    \end{align*}
    which yields the claimed inequalities. 
    
    Finally, we take $\psi(t) = e^{-\frac{t^2}{2}}$ and directly compute
    \[
    p\Mel{\tilde \psi}{p} = p\int_0^{\infty}e^{-\frac{t^2}{2}}t^{p-1}\;\dlat t= 2^{\frac{p}{2}}\Gamma\left(1+\frac{p}{2}\right).
    \]
    Therefore, by Stirling's approximation,
    \[
    \lim_{p\to\infty}\left(\frac{p\Mel{\tilde \psi}{p}}{\Gamma\left(p+1\right)}\right)^\frac{1}{p} =0.
    \]
     
\end{proof}
Next, we will need the following lemma to analyze the behaviour of the Mellin transform as $p\to -1$; see, e.g., \cite[Lemma 4]{HL25} for a proof.
\begin{lemma} 
\label{l:fractional_deriv}
If $\varphi:\R_+  \longrightarrow \R_+$ is a measurable function with $\lim _{t \rightarrow 0^{+}} \varphi(t)=$ $\varphi(0)$ and such that $\int_0^{\infty} t^{p_0} \varphi(t) \mathrm{d} t<\infty$ for some $p_0 \in(-1,0)$, then
$$
\lim _{p \rightarrow (-1)^{+}}(1+p) \int_0^{\infty} t^{p} \varphi(t) \mathrm{d} t=\varphi(0) .
$$
\end{lemma}\noindent
This procedure is very classical, and is sometimes referred to as \textit{fractional derivative}. See, for example, Koldobsky's monograph \cite{AK05} for further details.

\subsection{Ball bodies.}
\label{sec:ball_bodies}
We begin this section with two definitions.
\begin{definition}
\label{def:star}
    A set $L\subset \R^n$ is a \textit{star-shaped set} (with respect to the origin) if $[o,x]\subset L$ for all $x\in L$. Furthermore, $L$ is a \textit{star body} if it is a star-shaped set that is compact with non-empty interior, and if its radial function from \eqref{eq:radial_function} is continuous on $\R^n\setminus\{o\}$.
\end{definition}\noindent
Clearly, every convex body containing the origin is a star body.

\begin{definition}
\label{def:keith_ball_bodies}
Let $g:\R^n \longrightarrow \R_+$ be a bounded, measurable function. Then, its Ball body is the star-shaped set $K_p(g)$ whose radial function is given by
\begin{equation}
    \label{eq:kbb}
    \begin{split}
    \rho_{K_p(g)}(\theta)&= \left(p\Mel{\widetilde{g(\cdot\theta)}}{p}\right)^\frac{1}{p}
    \\
    &=\begin{cases} \left(\frac{p}{\|g\|_\infty}\int_0^\infty g(r\theta)r^{p-1}\, \dlat r\right)^\frac{1}{p}, &p>0,
    \\
    \exp\left(\frac{1}{\|g\|_\infty}\int_0^\infty(-\frac{\partial}{\partial r}g(r\theta))\log(r)\; \dlat r\right), &p=0,
    \\
    \left(\frac{p}{\|g\|_\infty}\int_0^\infty r^{p-1}(g(r\theta)-g(o))\; \dlat r\right)^\frac{1}{p}, & p\in (-1,0).
    \end{cases}
    \end{split}
\end{equation}
For the formula $p=0$, $\frac{\partial}{\partial r}$ denotes the one-sided derivative of $g(r\theta)$ in $r$, and we additionally require that this exists for almost all $\theta$ and $r$ when considering $p=0$.
\end{definition}
The condition concerning differentiability is satisfied when $g$ is $\log$-concave; its one-sided derivatives exist since the convex function $r\mapsto-\log g(r\theta)$ has one-sided derivatives everywhere on its domain. 

 For $p>0$, if $g$ has finite $(p-n)$th moment, then $\rho_{K_{p}(g)} \in L^{p}(\s)$, which, in-turn, yields $\vol_{n}(K_{p}(g)) < \infty$: indeed, from Jensen's inequality and an application of polar coordinates, we have
\begin{equation}
\label{eq:finite_vol}
\begin{split}
\frac{n}{p}\vol_n(\B)^{1-\frac{p}{n}}\vol_{n}(K_{p}(g))^\frac{p}{n} &\leq \frac{1}{p}\int_{\s}\!\rho_{K_{p}(g)}(\theta)^{p}\,\dlat\theta 
\\
&= \int_{\R^n}\!\!\!g(x)|x|^{p-n}\,\dlat x.
\end{split}
\end{equation}
Under the stronger assumption that $r\mapsto g(r\theta)$ has finite, strictly positive, $(p-1)$th moment for all $\theta\in\s$, or equivalently $0<\rho_{K_p(g)}<\infty$, the set $K_p(g)$ is then a star body. 

In particular, the bodies $K_p(g)$ were originally introduced by K. Ball \cite{Ball88} when $p>0$ and the function $g$ is $\log$-concave and even; Gardner and Zhang \cite{GZ98} later considered the case of not necessarily even, log-concave functions. We summarize their results in the following proposition, asserting that, if $f\in\lc$, then $K_p(f)$ is convex.
\begin{proposition}[Theorem 5 in \cite{Ball88} and Corollary 4.2 in \cite{GZ98}]
\label{p:radial_ball}
   If $f \in \lc$ then, for every $p> 0,$ the function on $\s$ given by 
    \begin{equation}
    \label{eq:keith_ball_body}
    \theta\mapsto\left(p\int_0^\infty f(r\theta)r^{p-1}\dlat \, r\right)^\frac{1}{p}\end{equation}
    defines the radial function of a convex body containing the origin, $K_p(f)$. If $f$ is even, then $K_p(f)$ is origin-symmetric.
\end{proposition}
When $p\in (-1,0)$, an analogous requirement to \eqref{eq:finite_vol} is that $g$ has maximum at the origin and the function $x\mapsto g(o)-g\left(\frac{x}{|x|^2}\right)$ has finite $(-p)$-moment. Indeed, from a change of variables
\begin{align*}
p \int_0^\infty r^{p-1}(g(r\theta)-g(o))\dlat \, r = |p|\int_{0}^{\infty}r^{|p|-1}\left(g(o)-g\left(\frac{\theta}{r}\right)\right)\dlat \, r,
\end{align*}
and then we can reduce to the case $p >0$.

We notice that, even though our extension of Ball's bodies to $p\in (-1,0]$ via the Mellin transform appears to be new, it was anticipated in the works \cite{FLM20,GZ98,KPY08}. It is unknown if $K_p(f)$ is convex when $f$ is log-concave for $p\in (-1,0)$. Nevertheless, it will still be a star body. As an application of our new definition, we show that the now classical bodies $R_p K$ fall into the same picture.

\begin{proof}[Proof of Proposition~\ref{p:equivalence}]
We first consider the case when $p\neq 0$.
The equivalence between \eqref{eq:radial_mean_og} and \eqref{radial_ell} is an application of Fubini's theorem. Indeed, for $p>0$, we have
\begin{align*}
\int_{K} \rho_{K-x}(\theta)^{p} \dlat \, x &= p\int_K\int_0^{\rho_{K-x}(\theta)}r^{p-1}\dlat \, r\,\dlat \, x
\\
&=p \int_{0}^{\rho_{D K}(\theta)}\left(\int_{K \cap(K+r \theta)} \dlat \, x\right) r^{p-1} d r 
\\
&=p \int_{0}^{\rho_{D K}(\theta)} g_{K}(r \theta) r^{p-1} d r,
\end{align*}
where, in the second step, we used the fact that $x\in K$ and $-r\theta\in K-x$ for all $0\leq r\leq \rho_{K-x}(\theta)$. Similarly, for $p\in (-1,0),$ we have 
\begin{align*}
&\int_{K} \rho_{K-x}(\theta)^{p} \dlat \, x =-p \int_K \int_{\rho_{K-x}(\theta)}^\infty  r^{p-1}\dlat \, r\dlat \, x
\\
&=-p\int_{0}^{\rho_{DK}(\theta)}\left(\int_{K\setminus {K\cap (K+r\theta)}}\dlat \, x\right)r^{p-1}\dlat \, r - p\int_K\int_{\rho_{DK}(\theta)}^\infty r^{p-1}\dlat \, r\dlat \, x.
\end{align*}
Adding and subtracting integration over $K\cap(K+r\theta),$ we obtain
\begin{align*}
\int_{K} \rho_{K-x}(\theta)^{p} \dlat \, x &=p\int_{0}^{\rho_{DK}(\theta)}(g_{K}(r\theta)-\vol_n(K))r^{p-1}\dlat \, r+\rho^p_{DK}(\theta)\vol_n(K)
\\
&=p\int_{0}^{\infty}(g_{K}(r\theta)-\vol_n(K))r^{p-1}\dlat \, r.
\end{align*}
From integration by parts, \eqref{radial_ell} re-writes as, for $p\neq 0$,
\begin{equation}\rho_{R_p K}(\theta)=\left(\int_0^{\rho_{DK}(\theta)}\frac{\partial}{\partial r}\left(-\frac{g_K(r\theta)}{\vol_n(K)}\right) r^{p}\dlat \, r\right)^\frac{1}{p}.
\label{eq:best_radial_form}
\end{equation}
Then, considering $p\to 0$, we obtain the second formula in \eqref{radial_ell}.

The fact that $R_p K$ is origin-symmetric follows the evenness of $g_K$. Finally, the convexity of $R_pK$ when $p>0$ follows from an application of Proposition~\ref{p:radial_ball}, which extends to $p=0$ via continuity. 
\end{proof}

Our next goal is to generalize \eqref{e:set_inclusion} to Ball bodies.  To do so, we define the following body, corresponding to $p=-1$.
\begin{proposition}
\label{p:limit_body}
    Let $g:\R^n\longrightarrow \R_+$ be a function that is log-concave on its support and obtains its maximum at the origin. Then, for all $\theta\in\s$,
    \[
    \lim_{p\to(-1)^+}(1+p)^\frac{1}{p}\rho_{K_p(g)}(\theta)= \left(-\frac{\partial}{\partial r}\frac{g(r\theta)}{\|g\|_\infty}\bigg|_{r=0^+}\right)^{-1},
    \]
    and, consequently, we define a star body $\tilde K_{-1}(g)$ by
    \[
    \rho_{\tilde K_{-1}(g)}(\theta) = \lim_{p\to(-1)^+}(1+p)^\frac{1}{p}\rho_{K_p(g)}(\theta)=\left(-\frac{\partial}{\partial r}\frac{g(r\theta)}{\|g\|_\infty}\bigg|_{r=0^+}\right)^{-1}.
    \]
\end{proposition}
We postpone the proof of Proposition~\ref{p:limit_body} to the appendix. We define the following constant for the asymptotics of the binomial coefficient \eqref{eq:binomial_coefficients} as $p$ goes to $-1$ from above, $$c(s)\coloneqq\lim_{p\to(-1)^{+}}\left(\frac{1}{1+p} \binom{p+\frac{1}{s}}{p}\right)^{\frac{1}{p}},$$
obtaining from direct computations the formula
\begin{equation}
    \label{eq:beta_constants}
    c(s)\coloneqq\begin{cases}
        \frac{1}{s}, & s>0,
        \\
        1,& s=0.
    \end{cases}
\end{equation}
The case $s>0$ can be computed by writing
\[
p\cdot B\left(p,\frac{1}{s}+1\right) = p\int_0^{1}(1+t)^\frac{1}{s}t^{p-1}\, \dlat t = \frac{1}{s}\int_0^{1}(1+t)^{\frac{1}{s}-1}t^p\, \dlat t
\]
and using Lemma~\ref{l:fractional_deriv} with $\varphi(t)=(1+t)^{\frac{1}{s}-1}$.

We also need a Ball body corresponding to $p=\infty$, which we establish in the next proposition.
\begin{proposition}
\label{p_infty_body}
    Let $g:\R^n\longrightarrow\R_+$ be a function that is log-concave on its support and obtains its maximum at the origin. Then,
    \[
    K_\infty(g) := \lim_{p\to\infty} K_p(g)= \supp(g),
    \]
    which is a potentially unbounded convex set.
\end{proposition}
\begin{proof}
    It suffices to show that $\lim_{p\to\infty}\rho_{K_p(g)}=\rho_{\supp(g)}$ point-wise on the sphere. We denote by $\theta\R_+$ the ray emanating from the origin towards $\infty$ in a given direction $\theta\in\s$. Consider the case when $\supp(g)\cap \theta\R_+$ is a line segment. Define the function $\psi(r)=\frac{g(r\theta)}{\|g\|_\infty}$, which is a compactly supported, log-concave function such that $\|\psi\|_\infty=1$. Then, by \eqref{eq:mellin_limit}
    \begin{align*}
\rho_{K_\infty(g)}(\theta)&=\lim_{p\to\infty} \rho_{K_p(g)}(\theta)=\lim_{p\to\infty} \left(p\Mel{\psi}{p}\right)^\frac{1}{p}
\\
&=\vol_1\left(\supp(g)\cap \theta\R_+\right)=\rho_{\supp(g)}(\theta).
\end{align*}
On the other hand, consider the case when $\supp(g)\cap \theta\R_+=\theta\R_+$. Then, letting $\psi(r)=g(r\theta)$, we have $\supp(\psi)=\R_+$; Proposition~\ref{p:p_to_infty} therefore yields
\[
\lim_{p\to\infty} \rho_{K_p(g)}(\theta) = \infty=\rho_{\supp(g)}(\theta).
\]
\end{proof}

Nevertheless, by  Proposition~\ref{p:p_to_infty}, we have 
\begin{equation}\lim_{p\to\infty} \Gamma(1+p)^{-\frac{1}{p}}\rho_{K_p(g)}(\theta)< \infty, \quad \forall \; \theta\in\s.
\label{eq:adjusted_radial_limit}
\end{equation}
Therefore,
\begin{equation}\lim_{p\to\infty} \Gamma(1+p)^{-\frac{1}{p}}K_p(g) 
\label{eq:adjusted_limit_sets}
\end{equation} 
will converge in the Hausdorff metric (compare \cite[Section 1.8]{Sh1}) to a compact, convex set. However, since the limit \eqref{eq:adjusted_radial_limit} pertaining to the radial function may be zero, the corresponding limit \eqref{eq:adjusted_limit_sets} of sets may converge to a singleton containing the origin.

With the necessary preparations completed, we have the following theorem, which is an immediate consequence of Propositions~\ref{p:increasing_mono} and \ref{p:mellin_berwald}. This theorem is well-known for $p,s \geq 0, p\neq \infty,$ (see e.g. \cite[Proposition 2.5.7]{BGVV14} and \cite[Lemma 1]{FMY17}). To the extent of our knowledge, this result is new for the choices $p\in (-1,0)$ and $p=\infty$.
\begin{theorem}
\label{t:ball_sets}
    Fix $s\geq 0$. Let $g:\R^n\longrightarrow\R_+$ be an integrable, $s$-concave function, obtaining its maximum at the origin. Then,
    \begin{equation}
    \label{eq:weak_chain}
    K_p(g) \subseteq K_q(g), \quad -1<p<q\leq\infty,
\end{equation}
with equality if and only if there exists a bounded, non-negative, function $a:\s\longrightarrow \R_+$ such that
\begin{enumerate}
    \item if $q>p>0$: $g(r\theta)=\|g\|_\infty\chi_{[0,a(\theta)]}(r)$, and,
    \item if $p<q<0$: $g(r\theta)=\|g\|_\infty\left(1-\chi_{[0,a(\theta)]}(1/r)\right)$.
\end{enumerate}
\vskip 1mm
Similarly, for $-1<p<q<\infty$, it holds
\begin{equation}
\label{eq:strong_chain}
\begin{split}
\lim_{p\to \infty}\left(\binom{p+\frac{1}{s}}{p}^{\frac{1}{p}}K_p(g)\right)&\subseteq\!\binom{q+\frac{1}{s}}{q}^{\frac{1}{q}}K_q(g) 
\\
&\subseteq \!\binom{p+\frac{1}{s}}{p}^{\frac{1}{p}}K_p(g) 
\!\subseteq \!c(s)\tilde K_{-1}(g),
\end{split}
\end{equation}
    with equality in any set inclusion if and only if there exists a bounded, non-negative, function $\alpha:\s\longrightarrow \R_+$ such that \begin{enumerate}
        \item if $s>0$: $g(r\theta)=\|g\|_\infty(1-\alpha(\theta)r)_+^\frac{1}{s}$, and
        \item if $s=0$: $g(r\theta)=\|g\|_\infty e^{-\alpha(\theta)r}\chi_{(0,\infty)}(r).$
    \end{enumerate}
\end{theorem}
\begin{proof}
The proof of the set-inclusions \eqref{eq:weak_chain} is an immediate application of Proposition~\ref{p:increasing_mono} with $\psi(r)=g(r\theta)$. 

As for the set-inclusions \eqref{eq:strong_chain}, we make the same choice of $\psi$ in Proposition~\ref{p:mellin_berwald} and obtain that the function
\begin{align*}
G(p)&= \left(\binom{p+\frac{1}{s}}{p}\left(p\Mel{\widetilde{g(\cdot\theta)}}{p}\right)\right)^\frac{1}{p} =\binom{p+\frac{1}{s}}{p}^{\frac{1}{p}} \rho_{K_p(g)}(\theta)
\\
&=\left(\frac{1}{1+p}\binom{p+\frac{1}{s}}{p}\right)^{\frac{1}{p}} (1+p)^\frac{1}{p}\cdot\rho_{K_p(g)}(\theta)
\end{align*}
is decreasing on $(-1,\infty)$,
and then the set-inclusions and equality characterization follow; the limiting behavior as $p\to -1$ is precisely from Proposition~\ref{p:limit_body} and the definition of the constant $c(s)$ from \eqref{eq:beta_constants}. Finally, Proposition~\ref{p:p_to_infty} and the discussion beforehand yield the case of $p\to \infty$.
\end{proof}

Theorem~\ref{t:ball_sets} provides new proofs of \eqref{eq:chain_og} and \eqref{e:set_inclusion}. Again applying Theorem~\ref{t:ball_sets}, we will generalize these inequalities to the setting of log-concave functions for an alternative proof of Theorem \ref{t:Zhang_fun_hi}. 

\section{The $m$th-Order, Functional Setting: \\ Zhang's Projection inequality}
\label{sec:hi}
We now establish the functional, $m$th-order Zhang's projection inequality, which we recall.

\begin{reptheorem}{t:Zhang_fun_hi}
    Fix $n,\hid\in\N.$ Let $f\in \lc$. Then, 
    \begin{equation}\frac{1}{(n\hid)!}\int_{\R^{n\hid}}g_{f,\hid}(\bar x)\,\dlat\bar x\leq \|f\|_1^{n\hid+1}\vol_{n\hid}(\PP \langle f \rangle).
\label{eq:Zhang_fun_hi}
\end{equation}
There is equality if and only if $f(x)=\|f\|_\infty e^{-\|x-x^\prime\|_{\Delta_n}}$ for an $n$-dimensional simplex $\Delta_n$ containing the origin and some $x^\prime\in\R^n$.
\end{reptheorem}
\subsection{The Polar Projection Body of a Function}
It will be convenient to represent $\PP \langle f \rangle$ in terms of $\PP\{f \geq t\}$. Note that, by definition, we have
\begin{equation}
    \|\bar x\|_{\PP \langle f \rangle} = \int_{\partial \langle f \rangle}h_{C_{-\bar x}}(n_{\langle f \rangle}(y))\,\dlat \mathcal{H}^{n-1}(y).
    \label{eq:PPLYZ}
\end{equation}
Hence, we infer $$\|\bar x\|_{\PP \langle e^{-\|x\|_K} \rangle}=(n-1)!\|x\|_{\PP K} \longrightarrow \PP \langle e^{-\|x\|_K} \rangle = \frac{1}{(n-1)!} \PP K.$$
In the general case,  we use \eqref{eq:PPLYZ} and the change of variables formula for the asymmetric LYZ body of $f\in \lc$, \eqref{eq:change_of_variable}, to obtain the formula \eqref{eq:hi_LYZ_BV}. Next, we apply the co-area formula \eqref{eq:coarea} to deduce
\begin{equation}
\label{eq:gauage_relate}
\begin{split}
    \|\bar x\|_{\PP \langle f \rangle} &=\int_0^{\|f\|_\infty}\int_{\partial \{f\geq t\}}h_{C_{-\bar x}}\left( n_{\{f \geq t\}}(y)\right)\,\dlat \mathcal{H}^{n-1}(y)\,\dlat t    \\
    &=\int_0^{\|f\|_\infty}\|\bar x\|_{\PP \{f\geq t\}}\,\dlat t.
\end{split}
\end{equation}

Before getting into the proof of Theorem \ref{t:Zhang_fun_hi}, let us point out some crucial properties of the functional covariogram.

\subsection{The Functional Covariogram}
\label{sec:fun_covario}
Let $f \in \lc$. Then, the functional covariogram of $f$ satisfies
\begin{equation}
\begin{split}
    g_{f,\hid}(\bar x)&=\int_{\R^n}\min_{0\leq i \leq m}\{f(y-x_i)\}\,\dlat y
     =\int_{\R^n}\int_{0}^{\min_{0\leq i \leq m}\{f(y-x_i)\}} \,\dlat t\,\dlat y 
    \\
    &=\int_0^{\infty}\int_{\R^n}\chi_{\{y:\min_{0\leq i \leq m}\{f(y-x_i)\} \geq t\}}(y)\dlat y \dlat t
    \\
    &=\int_0^{\infty}\int_{\R^n}\chi_{\{y: y-x_i\in\{f \geq t\}\}}(y)\dlat y \dlat t
     \\
    &=\int_0^{\infty}\int_{\R^n}\chi_{\{y\in \cap_{0\leq i \leq m}\left(\{f \geq t\}+x_i\right)\}}(y)\dlat y \dlat t =\int_0^\infty g_{\{f \geq t\},\hid}(\bar x) \,\dlat t.
\end{split}
\label{eq:covario_relate}
\end{equation}
That is, the functional covariogram can be described as the integration of the covariograms of the level sets of $f$.

\begin{remark}
If one were to apply the $m$th-order Zhang's projection inequality to the level sets of $f$ and integrate, they would obtain from \eqref{eq:covario_relate}

\begin{equation*}\frac{1}{(n\hid)!}\int_{\R^{n\hid}}g_{f,\hid}(\bar x)\,\dlat \bar x\leq \int_0^{\|f\|_\infty} \vol_{n\hid}(\PP \{f\geq t \}) \,\dlat t.
\end{equation*}

But notice that the above Zhang's projection-type inequality is weaker than an inequality involving the asymmetric LYZ body from Theorem~\ref{t:Zhang_fun_hi}. Indeed, by Jensen's inequality and \eqref{eq:gauage_relate}
\begin{align*}
    \|f\|_\infty^{n\hid}\|\bar x\|_{\PP \langle f \rangle}^{-n\hid} 
    &=\left(\int_0^{\|f\|_\infty}\|\bar x\|_{\PP \{f\geq t\}}\frac{\,\dlat t}{\|f\|_\infty}\right)^{-n\hid}
    \\
    & \leq \int_0^{\|f\|_\infty}\|\bar x\|^{-n\hid}_{\PP \{f\geq t\}}\frac{\,\dlat t}{\|f\|_\infty}.
\end{align*}
Integrating over the sphere, we deduce $$\|f\|_\infty^{n\hid+1}\vol_{n\hid}(\PP \langle f \rangle) \leq \int_0^{\|f\|_\infty} \vol_{n\hid}(\PP \{f\geq t \}) \,\dlat t.$$
\end{remark}

 Next, we analyze in more detail the support of $g_{f,\hid}$.
\begin{proof}[Proof of Proposition~\ref{p:support}]
    For the first claim, observe that $g_{f\hid}(\bar o) = \|f\|_{1}> 0$. Therefore, the assertion follows by continuity. We now turn our attention to the equality of sets. The first equality is an application of Proposition~\ref{p_infty_body} with $g=g_{f,\hid}$. Next, fix $y_0\in \R^n$ and suppose that $\supp(f)=\R^n$. Then, for a fixed $\bar{x}\in \R^{n\hid}$, We can find an Euclidean ball $B$ such that $y_0-x_i\in B$ for all $i=1,\dots,m$. Furthermore, on this ball, $f \geq a$ for some $a>0$. Therefore,
    \begin{align*}
         g_{f,\hid}(\bar x)&=\int_{\R^n}\min_{0\leq i \leq m}\{f(y-x_i)\}\,\dlat y 
         \\
         &\geq \int_{B}\min_{0\leq i \leq m}\{f(y-x_i)\}\,\dlat y 
         \geq a\vol_n(B) >0.
    \end{align*}
    From the arbitrariness of $\bar x$, this shows $\supp(g_{f,\hid})=\R^{nm}$.
    
    For the final claim, we use \eqref{eq:covario_relate}. We have that $g_{f,\hid}(\bar x) >0$ if and only if $g_{\{f \geq t\},\hid}(\bar x) >0$ on a set of positive measure $T\subset \R_+$ (in the variable $t$). This is true if and only if there exists $z$, depending on $t$,
    \begin{equation}
    \label{eq:f_sets}
    z\in \{f \geq t\}\cap\bigcap_{i=1}^m\left(\{f \geq t\} + x_i\right)
    \end{equation}
    for all such $t$, but this means each $z$ belongs to 
    \[
    \supp(f)\cap\bigcap_{i=1}^m\left(\supp(f) + x_i\right),
    \]
    i.e. $\bar x \in D^m(\supp(f)).$ Conversely, if $z\in \text{int}(D^m(K))$, then there exists a range of $t$ such that \eqref{eq:f_sets} holds, yielding $g_{\{f\geq t\},\hid}(\bar x) >0$ on a set of positive measure.
\end{proof}

Next, we take the radial derivative of the functional covariogram. First, we need the following lemma. Recall that, for a function $h$ on the sphere $\s$, the Wulff shape of $h$ is defined by $$[h]=\{x\in\R^n: \max_{u\in\s} \langle x,u\rangle \leq h(u)\}.$$
\begin{lemma}\label{lemma:leibniz}
    Let $f\in \lc$ and let $\bar{\theta} \in \S$. Then, for $\varepsilon>0$ suitably small,
    \begin{itemize}
        \item[(a)] For every $r \in [0,\varepsilon]$, $g_{\{f \geq t\},m}(r\bar{\theta})$ is integrable in $t \in [0, \| f\|_{\infty}]$,
        \item[(b)] The partial derivative \[\frac{\partial }{\partial r}g_{\{f \geq t\},m}(r \bar{\theta})\] exists for every $r \in [0,\varepsilon]$ and $t \in [0, \| f\|_{\infty})$,
        \item[(c)] There exists $c(t)$ integrable on $[0, \| f\|_{\infty}]$ such that \[ g_{\{f \geq t\},m}(r\bar{\theta})\leq c(t)\] for every $r \in [0,\varepsilon]$ and $t \in [0, \| f\|_{\infty}]$.
    \end{itemize}
\end{lemma}
\begin{proof}
    Notice that $(a)$ follows from \eqref{eq:covario_relate} and the fact that $f$ is integrable. Point $(c)$ follows choosing $c(t)=\vol_n (\{ f\geq t\})$. 

    We now focus on point $(b)$, which is slightly more delicate. Notice that, for a convex body $K$, we can write (cf. \cite[Proposition 3.2]{HLPRY25}) \[ \Cov(r\bar{\theta})=\vol_n\left( [h_K-rh_{C_{-\bar{\theta}}}]\right).\] By the semigroup property of the Wulff shape (compare \cite[Theorem 5.7]{Wi80}) and Aleksandrov's variational lemma \cite[Lemma 7.5.3]{Sh1}, we have for $\varepsilon>0$ sufficiently small that \[\frac{\partial}{\partial r}g_{\{f \geq t\},m}(r \bar{\theta})\] is right-continuous in $r$ on the interval $[0,\varepsilon]$. Since $\Cov(r\bar\theta)$ is continuous for $r\geq 0$ and has continuous right derivative on $[0,\varepsilon]$, it is differentiable on $[0,\varepsilon]$ (see, for example, \cite[Theorem 1.3, page 40]{Bruckner_Differentiation}). Notice that for every $t \in [0, \| f\|_{\infty})$ the set $\{f \geq t\}$ has non-empty interior so that we can apply the previous procedure to every level set (except, at most, the one corresponding to the maximum). This concludes the proof of $(b)$.
\end{proof}
As a consequence of Lemma \ref{lemma:leibniz}, we extend Matheron's identity \eqref{e:matheron} to the functional setting. We remark that the following result is, to the best of our knowledge, new even in the case of $m =1$.
\begin{theorem}
\label{t:co_fun_hi}
    Fix $m,n\in\N$. Let $f\in \lc$. Then, for every $\bar\theta\in \S$,
    $$\frac{\partial }{\partial r}g_{f,\hid}(r\bar\theta)\big|_{r=0^+} = -\|\bar \theta\|_{\PP \langle f \rangle}.$$
\end{theorem}
\begin{proof}
    Observe that \eqref{eq:covario_relate} and \eqref{eq:gauage_relate} yield
\begin{align*}
    \frac{\partial }{\partial r}g_{f,\hid}(r\bar\theta)\big|_{r=0^+}&=\int_{0}^\infty \frac{\partial}{\partial r}g_{\{f\geq t\},m}(r\bar\theta)\big|_{r=0+}\,\dlat t
    \\
    &=-\int_{0}^\infty\|\bar x\|_{\PP \{f\geq t\}}\,\dlat t=-\|\bar \theta\|_{\PP \langle f \rangle},
\end{align*}
where we could interchange the integral and derivative by Lemma \ref{lemma:leibniz}.
\end{proof}
As a corollary, we obtain the following "tangent-line bound" for the functional covariogram.
\begin{corollary}\label{c:log_cov}
    Fix $n,m\in\N$. Let $f\in \lc$. Then, $g_{f,\hid}$ is log-concave and, for a fixed $\bar x \in \R^{n\hid}$, 
    \begin{equation} g_{f,\hid}(\bar x) \leq \|f\|_1e^{-\frac{\|\bar x\|_{\PP \langle f \rangle}}{\|f\|_1}},
\label{eq:log_cov}
\end{equation}
with equality if and only if the function $r\mapsto g_{f,\hid}(r\bar x)$ is $\log$-affine on its support.
\end{corollary}
\begin{proof}
    First, observe that $g_{f,\hid}$ is log-concave. Indeed, the minimum of log-concave functions is log-concave (since the maximum of convex functions is convex). Then, from the Pr\'ekopa-Leindler \eqref{eq:PL} inequality applied in $\R^{n\hid}$, we obtain the log-concavity of $g_{f,\hid}$. 

    Next, as $r\mapsto g_{f,\hid}(r\bar\theta)$ is log-concave, we obtain from Theorem~\ref{t:co_fun_hi} that, for every $r>0$ such that $r\bar\theta$ is in the support of $g_{f,\hid}$,
\begin{align*}-\frac{\|\bar \theta\|_{\PP \langle f \rangle}}{\|f\|_1}&=\frac{\partial}{\partial r}\log g_{f,\hid}(r\bar\theta)\Big|_{r=0^+} 
\\
&\geq \frac{\log g_{f,\hid}(r\bar\theta) - \log\|f\|_1}{r}=\frac{\log\left(\frac{g_{f,\hid}(r\bar\theta)}{\|f\|_1}\right)}{r}.\end{align*}
Consequently, we obtain
$$r\|\bar\theta\|_{\PP \langle f \rangle} \leq -\|f\|_1\log\left(\frac{g_{f,\hid}(r\bar\theta)}{\|f\|_1}\right).$$

This is precisely our claim after exponentiation.
\end{proof}
\noindent In the following lemma, we show that equality occurs in Corollary~\ref{c:log_cov} if and only if $f$ is a multiple of $e^{-\|x-x^\prime\|_{\Delta_n}}$ for some $x^\prime \in \R^n$ and $n$-dimensional simplex $\Delta_n$ containing the origin.
\begin{lemma}
\label{l:equality_cond}
    Suppose $f\in \lc$ is such that, for almost all $\bar\theta\in\mathbb{S}^{nm-1}$, the function $r\mapsto g_{f,\hid}(r\bar\theta)$ is $\log$-affine. Then, $f(x)=\|f\|_\infty e^{-\|x-x^\prime\|_{\Delta_n}}$ for an $n$-dimensional simplex $\Delta_n$ containing the origin and some $x^\prime \in \R^n$.
\end{lemma}
\begin{proof}
Notice that one may verify via direct substitution that, if $f(x)=\|f\|_\infty e^{-\|x-x^\prime\|_{\Delta_n}}$ for an $n$-dimensional simplex $\Delta_n$ containing the origin and some $x^\prime \in \R^n$, then $r\mapsto g_{f,\hid}(r\bar \theta)$ is $\log$-affine. It remains to show the converse direction.

Since $g_{f,\hid}$ is invariant under translations of $f$, we may assume that $\|f\|_\infty=f(o)$. For $x\in\R^n$, let $\tilde x\in\R^{n\hid}$ be $\tilde x=(x,o,\dots,o)$. Then, $g_{f,\hid}(\tilde x) = g_{f}(x)$. It is precisely the content of \cite[Lemma 2.4.2]{ABG20} that, if $f(o)=\|f\|_\infty$ and $r\mapsto g_{f}(r\theta)$ is $\log$-affine on its support for almost all $\theta\in\s$, then $f(x)=\|f\|_\infty e^{-\|x\|_{\Delta_n}}$ for some $n$-dimensional simplex $\Delta_n$ containing the origin.
\end{proof}

\label{sec:radial}
We finish by showing how Corollary \ref{c:log_cov} reduces Theorem \ref{t:Zhang_fun_hi} to the following simple proof. 
\begin{proof}[Proof of Theorem~\ref{t:Zhang_fun_hi}]
Observe that, integrating both sides of \eqref{eq:log_cov} over $\R^{n\hid}$ yields
\begin{align*}\frac{1}{(n\hid)!}\int_{\R^{n\hid}}g_{f,\hid}(\bar x)\,\dlat \bar x &\leq \|f\|_1\frac{1}{(n\hid)!}\int_{\R^{n\hid}}e^{-\frac{\|\bar x\|_{\PP \langle f \rangle}}{\|f\|_1}}\,\dlat \bar x
\\
&=\|f\|_1^{n\hid+1}\frac{1}{(n\hid)!}\int_{\R^{n\hid}}e^{-\|\bar x\|_{\PP \langle f \rangle}}\,\dlat\bar x
\\
&=\|f\|_1^{n\hid+1}\vol_{n\hid}(\PP \langle f \rangle).
\end{align*}
with equality if and only if $g_{f,\hid}$ is $\log$-affine. Equality conditions then follow from Lemma~\ref{l:equality_cond}.
\end{proof}

\subsection{The approach by radial mean bodies}
\label{sec:radial_function_bodies}
As we anticipated above, our second main goal is to explore functional analogues of radial mean bodies. To the best of our knowledge, such analogues have not been previously studied, even in the case $m=1$. In this regard, we employ the framework established in Section \ref{sec:class}. We start by introducing the definition of this family of sets.
\begin{definition}
    Fix $m,n\in\N$ and let $f \in \lc$. The $m$th-order, functional radial $p$th mean body of $f$ is the star body $R^\hid_p f \subset \R^{n\hid}$ given by the radial function: for $\bar\theta\in\S,$
\begin{equation}\rho_{R^\hid_p f}(\bar\theta)\!=\!\begin{cases}
    \left(\frac{p}{\|f\|_1}\int_0^{\infty}\!g_{f,\hid}(r\bar\theta)r^{p-1}\,\dlat r\right)^\frac{1}{p}, &\! p>0,
    \\
    \left(\frac{p}{\|f\|_1}\!\int_{0}^{\infty}\! \left(g_{f,\hid}(r\bar\theta)\!-\!\|f\|_1\right) \!r^{p-1} \,\dlat r\right)^\frac{1}{p}, & \!p\in (-1,0).
\end{cases}
\label{eq:radial_fun_hi}
\end{equation}
\end{definition}
Clearly, the definition is well-posed, since the restriction of a summable log-concave function on $\R^n$ to a line is a summable one-dimensional log-concave function. Define the function $f_{y,m}:\R^{nm}\longrightarrow \R_+$ by
$$f_{y,m}(\bar x) = \min_{0\leq i \leq m}\left\{f(y-x_i)\right\},$$
where $y\in\R^n$ and $f:\R^n\longrightarrow\R_+$ are fixed. Then, for $p\neq 0$,
\[
\rho_{R^\hid_p f}(\bar\theta) = \left(\frac{\|f\|_\infty}{\|f\|_1}\int_{\R^n}\rho_{K_p(f_{y,m})}^p(\bar\theta)\,\dlat y\right)^\frac{1}{p},
\]
and then $p=0$ follows by continuity. This should be compared to \eqref{eq:radial_mean_og}. Notice that we can apply integration by parts to \eqref{eq:radial_fun_hi} to obtain
\begin{equation}
\label{eq:radial_uni}
    \rho_{R^\hid_p f}(\bar\theta)\!=\!\begin{cases} \left(\int_0^\infty\left(-\frac{\partial}{\partial r}\frac{g_{f,\hid}(r\bar\theta)}{\|f\|_1}\right) r^p \,\dlat r\right)^\frac{1}{p}, & p\!\in\!(-1,0)\!\cup\!(0,\infty),
    \\
    \exp\left(\int_0^\infty\!\left(\!-\frac{\partial}{\partial r}\!\frac{g_{f,\hid}(r\bar\theta)}{\|f\|_1}\right) \log(r)\,\dlat r\right), & p=0,
    \end{cases}
\end{equation}
where we used Lebesgue's theorem to obtain that $g_{f,\hid}(r\bar \theta)$ is differentiable almost everywhere on its support (in $r$) as it is monotonically decreasing in the variable $r.$ The formula for $p=0$ follows by continuity. 

Once \eqref{eq:radial_uni} is established, Theorem \ref{t:radial_hi_fun_set} is merely an application of Theorem~\ref{t:ball_sets} with $g=g_{f,\hid},$ and $s=0$. For completeness, we present a proof from base principles. 
\begin{proof}[Proof of Theorem~\ref{t:radial_hi_fun_set}]
   Notice that for a fixed $\bar\theta$ the function \[r\mapsto \|f\|_1^{-1} g_{f,\hid}(r\bar\theta)\] is non-increasing. Furthermore, we recall that $g_{f,\hid}(r\bar\theta)$ is $\log$-concave and thus, from Proposition~\ref{p:mb}, the function
    $$p\mapsto \left(\frac{1}{\Gamma(p+1)}\right)^\frac{1}{p} \rho_{R^\hid_p f}(\bar\theta)$$ is non-decreasing, and is constant if and only if for a fixed $\bar\theta$ \begin{equation} g_{f,\hid}(r\bar\theta) = \|f\|_1e^{-cr}
    \label{eq:equal}
    \end{equation}
    for some $c>0$, establishing our claim. 
    
    We now discuss the limiting behaviour in $p$ (and determine $c$). To determine the behaviour as $p\to -1$, observe we can write
$$\frac{1}{\Gamma(p+1)}\rho_{R^\hid_p f}(\bar\theta)^p = \frac{1}{\Gamma(p+2)}[(p+1)\rho_{R^\hid_p f}(\bar\theta)^p].$$
We then obtain from \eqref{eq:radial_uni} and Lemma~\ref{l:fractional_deriv} that
\begin{align*}\lim_{p\to -1} \frac{1}{\Gamma(p+2)}[(p+1)\rho_{R^\hid_p f}(\bar\theta)^p] &= -\frac{1}{\|f\|_1}\frac{\partial}{\partial r}g_{f,\hid}(r\bar\theta)\big|_{r=0^+} 
\\
&= \frac{\|\bar \theta\|_{\PP \langle f \rangle}}{\|f\|_1}.
\end{align*}

Now, again, suppose there is equality. In this case, when determining the behaviour as $p\to -1$, calculate the derivative of $g_{f,\hid}$ using \eqref{eq:equal} to conclude
\[c= \frac{\|\bar \theta\|_{\PP \langle f \rangle}}{\|f\|_1}.\]
The equivalent form of the equality conditions follows from Lemma~\ref{l:equality_cond}.
\end{proof}

As a result of Theorem \ref{t:radial_hi_fun_set}, we present the following alternative approach to Theorem \ref{t:Zhang_fun_hi}.
\begin{proof}[Alternative proof of Theorem~\ref{t:Zhang_fun_hi}]
    By integrating in polar coordinates, we obtain that
\begin{align*}\frac{1}{\|f\|_1}&\int_{\R^{n\hid}}g_{f,\hid}(\bar x)\,\dlat \bar x=\frac{1}{\|f\|_1}\int_{\S}\!\int_0^{\infty}\!g_{f,\hid}(r\bar \theta)r^{n\hid-1}\,\dlat r\,\dlat \bar\theta
\\
&=\frac{1}{n\hid}\int_{\S}\rho_{R^\hid_{n\hid} f}(\bar\theta)^{n\hid}\,\dlat \bar\theta=\vol_{n\hid}(R^\hid_{n\hid} f).
\end{align*}
Consequently, Theorem~\ref{t:Zhang_fun_hi} follows from Theorem~\ref{t:radial_hi_fun_set} by setting $p=n\hid$ and considering the volume of both sides.
\end{proof}

We finish this section by characterizing the radial mean bodies of a particular class of functions which contains those of Corollary \ref{cor:zero}. Recall that $R^m_p K=K_p (g_{K,m})$, i.e. its radial function is similar to \eqref{eq:radial_fun_hi} and \eqref{eq:radial_uni}, but with $\|f\|_1$ and $g_{f,m}$ replaced with $\vol_n(K)$ and  $g_{K,m}$ respectively.
\begin{theorem}
\label{t:bodies_relate}
Fix $m,n\in\N$. Let $\varphi:\R_+\longrightarrow\R_+$ be a measurable function with finite $(n-1)$th moment. Suppose further that $\varphi$ has maximum at the origin and decreases monotonically to zero at $\infty$. Fix $K\subset \R^n$ a convex body and define the function $f:\R^n\longrightarrow\R_+$ by $f=\varphi(\|\cdot\|_K)$. 

\noindent Then, if $p>0$ and $\varphi$ also has finite $(n+p-1)$th moment, one has
\[
R^\hid_p f = \left(\frac{n+p}{n}\frac{\int_{0}^{\infty}\varphi(t)t^{n+p-1}\dlat t}{\int_{0}^{\infty}\varphi(t)t^{n-1}\;\dlat t}\right)^\frac{1}{p}R_p^{\hid}K,
\]
while, if $p=0$ and $\varphi(t)t^{n-1}\log(t)$ is integrable on $\R_+$, one obtains
\[
R^\hid_0 f = e^\frac{1}{n}\exp\left(\frac{\int_0^{\infty}\varphi(t)t^{n-1}\log (t)\;\dlat t}{\int_{0}^{\infty}\varphi(t)t^{n-1}\;\dlat t}\right)R_0^{\hid}K.
\]
\end{theorem}
\begin{proof}
    First, we compute $g_{f,m}$. Since $\varphi$ is monotone decreasing, it is differentiable almost everywhere. Therefore, for every $\bar x\in \R^{nm},$ it holds 
    \begin{align*}
        g_{f,m}(\bar x) &= \int_{\R^n}\min_{0\leq i \leq m}\varphi(\|y-x_i\|_K)\;\dlat y = \int_{\R^n}\varphi\left(\max_{0\leq i \leq m}\|y-x_i\|_K\right)\;\dlat y 
        \\
        &= \int_{\R^n}\int_{\max_{0\leq i \leq m}\|y-x_i\|_K}^{\infty}(-\varphi^\prime(t))\; \dlat t \;\dlat y
        \\
        &= \int_{0}^{\infty}(-\varphi^\prime(t))\vol_{n}\left(tK\cap \bigcap_{i=1}^m\left(tK+x_i\right)\right)\; \dlat t
         \\
        &= \int_{0}^{\infty}(-\varphi^\prime(t))t^{n}\vol_{n}\left(K\cap \bigcap_{i=1}^m\left(K+t^{-1}x_i\right)\right)\; \dlat t
        \\
        &=\int_0^{\infty}(-\varphi^\prime(t))t^{n}g_{K,m}(t^{-1}\bar x) \; \dlat t.
    \end{align*}
    Next, we compute $g_{f,m}(\bar o) = \|f\|_1$:
    \begin{equation}
    \label{eq:g_f_0}
    \begin{split}
        \int_{\R^n}\varphi(\|x\|_K)\; \dlat x &= \int_{\s}\int_{0}^{\infty}\varphi(r\|\theta\|_K)r^{n-1}\, \dlat r\; \dlat \theta 
        \\
        &= \left(n\int_{0}^{\infty}\varphi(r)r^{n-1}\;\dlat r\right)\left(\frac{1}{n}\int_{\s}\|\theta\|_K^{-n}\; \dlat \theta\right)
        \\
        &= \vol_n(K)\left(n\int_{0}^{\infty}\varphi(r)r^{n-1}\;\dlat r\right).
    \end{split}
    \end{equation}
    We compute the radial function of $R^\hid_p f$ when $p>0$: for every $\bar\theta\in\mathbb{S}^{nm-1}$,
    \begin{align*}
        &\rho_{R^\hid_p f}(\bar\theta)^p=\frac{p}{\|f\|_1}\int_0^{\infty} g_{f,m}(r\bar \theta)r^{p-1}\; \dlat r
        \\
        &=\frac{p}{\|f\|_1}\int_{0}^{\infty}\int_{0}^{\infty}(-\varphi^\prime(t))t^{n}g_{K,m}(rt^{-1}\bar \theta)r^{p-1}\;\dlat t\;\dlat r
        \\
        &=\frac{p}{\|f\|_1}\int_{0}^{\infty}\int_{0}^{\infty}(-\varphi^\prime(t))t^{n+p}g_{K,m}(r\bar \theta)r^{p-1}\;\dlat t\;\dlat r
        \\
        &=\left(\frac{p}{\vol_n(K)}\int_{0}^{\infty}g_{K,m}(r\bar\theta)r^{p-1}\dlat r\right)
        \left(\frac{\vol_n(K)}{\|f\|_1}\int_{0}^{\infty}(-\varphi^\prime(t))t^{n+p}\dlat t\right)
        \\
        &=\rho_{R^{m}_p K}(\bar\theta)^p\left(\frac{\int_{0}^{\infty}(-\varphi^\prime(t))t^{n+p}\dlat t}{n\int_{0}^{\infty}\varphi(t)t^{n-1}\;\dlat t}\right).
    \end{align*}
    
    The claim follows since, by integration by parts, \[
    \int_{0}^{\infty}(-\varphi^\prime(t))t^{n+p}\dlat t = (n+p)\int_{0}^{\infty}\varphi(t)t^{n+p-1}\;\dlat t.
    \]

    For $p=0$, we first take the partial derivative of $g_{f,m}(\bar x) $: since $r\mapsto g_{K,m}(r\bar\theta)$ is $(1/n)$-concave, it is bounded from above by an $(1/n)$-affine function, specifically $g_{K,m}(r\bar\theta)\leq \vol_n(K)\left(1-\frac{r\|\bar\theta\|_{\PP K}}{n\vol_n(K)}\right)_+^n$. Therefore, we may apply the Leibniz rule and obtain, for every $r>0$ and $\bar\theta\in\S,$ 
    \begin{equation}
    \label{eq:partial_g_f_deriv}
    \frac{\partial}{\partial r}g_{f,m}(r\bar \theta) = \int_0^{\infty}(-\varphi^\prime(t))t^{n}\left(\frac{\partial}{\partial r}g_{K,m}(t^{-1}r\bar \theta)\right) \; \dlat t.
    \end{equation}
    Inserting our choice of $f$ into \eqref{eq:radial_uni} when $p=0$,  \eqref{eq:partial_g_f_deriv} and \eqref{eq:g_f_0} yield
     \begin{align*}
     &\rho_{R^\hid_p f}(\bar\theta) =\exp\left(\int_0^\infty\left(-\frac{\partial}{\partial r}\!\frac{g_{f,\hid}(r\bar\theta)}{\|f\|_1}\right) \log(r)\,\dlat r\right)
     \\
     &=\exp\left(\int_0^{\infty}\left(\frac{(-\varphi^\prime(t))t^{n}}{n\int_{0}^{\infty}\varphi(t)t^{n-1}\;\dlat t}\right)\int_0^\infty\left(-\frac{\partial}{\partial r}\frac{g_{K,m}(t^{-1}r\bar \theta)}{\vol_n(K)}\right) \log(r)\,\dlat r\dlat t\right).
     \end{align*}
     We next introduce the variable substitution $r=st$. Adjusting both the integral and derivative in $r$, we obtain, by inserting $\rho_{R_0^m K}$ and using
     \[
     \int_0^\infty\left(-\frac{\partial}{\partial s}\frac{g_{K,m}(s\bar \theta)}{\vol_n(K)}\right) \;\dlat s =1,
     \]
     the formula
     \begin{align*}
     &\rho_{R^\hid_p f}(\bar\theta)
     \!=\!\exp\left(\!\int_0^{\infty}\!\!\!\left(\frac{(-\varphi^\prime(t))t^{n}}{n\int_{0}^{\infty}\varphi(t)t^{n-1}\;\dlat t}\right)\int_0^\infty\left(-\frac{\partial}{\partial s}\frac{g_{K,m}(s\bar \theta)}{\vol_n(K)}\right) \log(st)\,\dlat s\dlat t\right)
      \\
     &=\exp\left(\frac{\int_0^{\infty}(-\varphi^\prime(t))t^{n}\log(t)\;\dlat t}{n\int_{0}^{\infty}\varphi(t)t^{n-1}\;\dlat t}\right)\exp\left( \left(\frac{\int_0^{\infty}(-\varphi^\prime(t))t^{n}\; \dlat t}{n\int_{0}^{\infty}\varphi(t)t^{n-1}\;\dlat t}\right)\log\rho_{R_0^\hid K}(\bar\theta)\right).
     \end{align*}

Finally, by integration by parts, we obtain
   \[
\rho_{R^\hid_0 f}(\bar\theta) = e^\frac{1}{n}\exp\left(\frac{\int_0^{\infty}\varphi(t)t^{n-1}\log (t)\;\dlat t}{\int_{0}^{\infty}\varphi(t)t^{n-1}\;\dlat t}\right)\rho_{R_0^\hid K}(\bar\theta).
\]
\end{proof}

A few remarks are in order. Firstly, if $f=e^{-\|\cdot\|_K},$ i.e. $\varphi(t)=e^{-t}$, we obtain $R_0^\hid f = e^{\psi(n+1)}R_0^\hid K$, and, for $p>0$
\[
R^\hid_p f = \left(\frac{\Gamma(n+p+1)}{\Gamma(n+1)}\right)^\frac{1}{p}R_p^{\hid}K.
\]
Using $\Gamma(n+1)^{-\frac{1}{p}}\to 1$ as $p\to \infty$, we have by Stirling's approximation
\begin{align*}
&\lim_{p\to \infty}\Gamma\left(1+p\right)^{-\frac{1}{p}}R^\hid_p f 
\\
&= \lim_{p\to\infty}\left(\frac{\Gamma(n+p+1)}{\Gamma(n+1)\Gamma(p+1)}\right)^\frac{1}{p}R_p^{\hid}K
= R_\infty^\hid K = D^m(K).
\end{align*}
Next, Corollary~\ref{cor:zero} then follows from Theorem~\ref{t:bodies_relate} with the choice $\varphi(t)=e^{-\frac{t^2}{2}}$ and an application of Stirling's approximation.

Finally, replicating the proof of Theorem~\ref{t:bodies_relate} when $p\in (-1,0)$ yields
\begin{align*}
        &\rho_{R^\hid_p f}(\bar\theta)
        \\
        &=\left(\!\frac{p}{\vol_n(K)}\!\int_{0}^{\infty}\!\!\!\left(\!\!\!\left(\!\frac{n+p}{n}\frac{\int_{0}^{\infty}\varphi(t)t^{n+p-1}\dlat t}{\int_{0}^{\infty}\varphi(t)t^{n-1}\;\dlat t}\right)\!g_{K,m}(r\bar \theta)-\!\!\vol_n(K)\!\!\right)r^{p-1}\dlat r\right)^\frac{1}{p}.
\end{align*}
The most direct and seemingly only way to relate this to $R^m_p K$ is to choose $\varphi(t)=Ce^{-\frac{n}{|p|}t^{|p|}}$ for some $C>0$, in which case the ratio of terms adjacent to $g_{K,m}$ is $1$. We outline this choice as a corollary of Theorem~\ref{t:bodies_relate}. We use that $C$ is a free parameter and set $C=e^{\frac{n}{|p|}}$, creating $\varphi(t)=e^{-\frac{n}{|p|}\left(t^{|p|}-1\right)}$. As $p\to 0$, this function converges to $t^{-n}$. We see that this limiting function does not satisfy the integrability requirements of the $p=0$ case of Theorem~\ref{t:bodies_relate}. Therefore, in the following corollary, we define a family of functions $f_{p,K}=\varphi\left(\|\cdot\|_K\right)$ and then arrive at $p=0$ by taking the limit as $p\to 0^-$ and as $p\to 0^+$, deriving in this way two bodies. This is because the formula for the coefficients has a jump discontinuity at $p=0$.

\begin{corollary}
      Fix $m,n\in\N$. Let $K\subset \R^n$ be a convex body and set \[
      f_{p,K}(x)=\begin{cases}
          e^{-\frac{n}{|p|}\left(\|x\|_K^{|p|}-1\right)}, & p >-1, p\neq 0,
          \\
          \|x\|_K^{-n}, & p=0.
      \end{cases}\]
      Then,
\[
R^\hid_p f_{p,K}= \begin{cases}
    R_p^\hid K, & p\in (-1,0),
    \\
    R_0^\hid K, & p=0^-,
    \\
    e^{1/n}R_0^\hid K, & p=0^+,
    \\
    \left(1+\frac{p}{n}\right)^\frac{1}{p} R_p^\hid K, & p>0.
\end{cases}
\]
\end{corollary}

\section{The Rogers-Shephard inequality in \\ the Functional, $m$th-Order Setting}
\label{sec:rs}
In order to establish the functional, $m$th-order Rogers-Shephard inequality, some definitions are in order; these definitions generalize the $\hid=1$ case in \cite{AlGMJV}. 

Recall the following notions from the introduction: given $\hid+1$ functions $f_0,\dots,f_\hid$ on $\R^{n}$, consider the vector-valued function $\bar f:\R^{n(\hid+1)}\longrightarrow\R^{\hid+1}$ as $$\bar f((x_0,\bar x)) = (f_0(x_0),f_1(x_1),\dots,f_\hid(x_\hid)).$$ Next, we define
\[
\begin{split}
    \mathcal{A}^{\hid}_{t}(\bar f)(\bar{x}) 
    &= \left\{z \in \supp(f_0) \cap_{i=1}^{m}(x_i - \supp(f_i)) \subset \R^n : \right.
    \\
    &\left. f_0(z)f_1(z - x_1) \cdot \cdot \cdot f_\hid(z - x_\hid) \geq t \Pi_{i=0}^\hid \|f_i\|_\infty \right\}.
    \end{split}
    \]
    Then, the $m$th-order convolution body of $\bar f$ is 
    $$
    \mathcal{C}^{\hid}_{\theta,t}(\bar f) = \{ \bar{x}\in\R^{n\hid} : \mathcal{A}^{\hid}_{t}(\bar f)(\bar{x}) \neq \emptyset, \vol_n(\mathcal{A}^{\hid}_{t}(\bar f)(\bar{x})) \geq \theta M_t \},
    $$
    where $M_t = \max_{\bar{x}} \vol_n(\mathcal{A}^{\hid}_{t}(\bar f)(\bar{x}))$. 
    
    In the case when we have, for all $i=0,\dots,\hid$, that $f_i =f$ for some function $f$, we write $\mathcal{A}^{\hid}_{t}(f)(\bar{x}) =  \mathcal{A}^{\hid}_{t}((f,\dots,f))(\bar{x})$. Then, we introduced the following two operations for these vector-valued maps, yielding, in turn, functions on $\R^{n\hid}$:
$$
(\bar f)_{\oplus_{\hid}} (\bar x) = \int_{\R^n} f_0(z)f_1(z - x_1) \cdot \cdot \cdot f_m(z - x_\hid) \dlat z
$$
and
$$
(\bar f)_{\star_{\hid}} (\bar{x}) = \sup_{z \in \R^n} f_0(z)f_1(z - x_1) \cdot \cdot \cdot f_m(z - x_m).
$$
We now recall the Rogers-Shephard-type inequalities from the introduction, whose proofs will take the remainder of this section. 

\begin{reptheorem}{t:hi_fun_RS}
     Fix $n,\hid\in\N$. Let $f_0,f_1,...,f_m\in \lc$. Then, setting $\bar f = (f_0,f_1,\dots,f_\hid):$
    \begin{align*}
        \|(\bar f)_{\oplus_{\hid}}\|_\infty\| (\bar f)_{\star_{\hid}}\|_{1}
        \leq \binom{n(\hid + 1)}{n} \left(\prod_{i=0}^\hid\|f_i\|_\infty\|f_i\|_1\right).
    \end{align*}
\end{reptheorem}

\begin{reptheorem}{t:hi_fun_RS_one}
    Fix $n,\hid\in\N$. Let $f\in \lc$ and set $$\bar{f}(\cdot) \coloneqq (f(\cdot),f(-\cdot),\cdots,f(-\cdot)).$$ Then:
    \begin{equation*}
       \|(\bar{f})_{\star_{\hid}}\|_1 \leq \binom{n(\hid + 1)}{n} \|f\|_\infty^m \|f\|_\frac{1}{m}.
    \end{equation*}
   There is equality if and only if $f/\|f\|_\infty$ is the characteristic function of an $n$-dimensional simplex.
\end{reptheorem}

We start by establishing two facts that will be used in the proofs. Although they are the generalization to our setting of Lemma 3.1 and Corollary 3.2 in \cite{AlGMJV}, we provide a proof for completeness. First, we show the following set inclusion.
\begin{lemma}\label{l:theta_convolution}
   Fix $m,n\in\N$. Let $f_0,\dots,f_\hid \in \lc$ and $t \in (0,1]$ be so that $M_t=\vol_n(\mathcal{A}^{\hid}_{t}(\bar f)(\bar{o}))$. Take $\theta_1, \theta_2, \lambda_1, \lambda_2 \in [0,1]$ such that $\lambda_1 + \lambda_2 \leq 1$. Then: 
   \begin{equation*}
       \lambda_1\mathcal{C}^{\hid}_{\theta_1,t}(\bar f) + \lambda_2\mathcal{C}^{\hid}_{\theta_2,t}(\bar f) \subseteq \mathcal{C}^{\hid}_{\theta,t}(\bar f),
   \end{equation*}
   with $ 1- \theta^{1/n} = \lambda_1 (1 - \theta_{1}^{1/n}) +  \lambda_2 (1 - \theta_{2}^{1/n})$.
\end{lemma}
\begin{proof}
    Let $\bar{x} \in {C}^{\hid}_{\theta_1,t}(\bar f)$ and $\bar{y} \in {C}^{\hid}_{\theta_2,t}(\bar f)$. Consider $z_0 \in  \mathcal{A}^{\hid}_{t}(\bar f)(\bar o), z_1 \in  \mathcal{A}^{\hid}_{t}(\bar f)(\bar{x})$ and $z_2 \in  \mathcal{A}^{\hid}_{t}(\bar f)(\bar{y})$. Then, we have that
    $$
    \prod_{i = 0}^{\hid}f_i(z_0)^\hid \geq t \prod_{i = 0}^{\hid}\|f_i\|_\infty,
    $$
    $$
    f_0(z_1) \prod_{i = 1}^{\hid}f_i(z_1-x_i) \geq t \prod_{i = 0}^{\hid}\|f_i\|_\infty, \quad \text{and}
    $$
    $$
    f_0(z_2) \prod_{i = 1}^{\hid}f_i(z_2-y_i) \geq t \prod_{i = 0}^{\hid}\|f_i\|_\infty.
    $$
Moreover, since the $f_i$ are log-concave functions, we deduce that
\begin{equation*}
    \begin{split}
        &f_0\bigl((1 - \lambda_1 - \lambda_2)z_0 + \lambda_1 z_1 +  \lambda_2 z_2\bigr)
        \\
        &\quad\quad\times\prod_{i = 1}^{\hid}f_i\bigl((1 - \lambda_1 - \lambda_2)z_0 - \lambda_1 (z_1-x_i)-\lambda_2(z_2-y_i)\bigr)\\
        &\geq \left(\prod_{i = 0}^{\hid}\!\!f_i(z_0)\!\!\right)^{(1 - \lambda_1 - \lambda_2)}\left(\!f_0(z_1)\!  \prod_{i = 1}^{\hid}\!f_i(z_1-x_i)\!\!\right)^{\lambda_1} \left(\!\!f_0(z_2)\!\!\prod_{i = 1}^{\hid}f_i(z_2-y_i)\!\!\right)^{\lambda_2}\\
        &\geq t \prod_{i = 0}^{\hid}\|f_i\|_\infty.
    \end{split}
\end{equation*}
Therefore, 

\begin{align*}
 \mathcal{A}^{\hid}_{t}(\bar f)(\lambda_1\bar{x} + \lambda_2 \bar{y}) &\supseteq (1 - \lambda_1 - \lambda_2)\mathcal{A}^{\hid}_{t}(\bar f)(\bar o) 
 \\
 &\quad\quad+ \lambda_1 \mathcal{A}^{\hid}_{t}(\bar f)(\bar{x}) + \lambda_2 \mathcal{A}^{\hid}_{t}(\bar f)(\bar{y}).
 \end{align*}

Using the Brunn-Minkowski inequality \eqref{e:BM}, we then obtain 
\begin{equation*}
    \begin{split}
        \vol_n( \mathcal{A}^{\hid}_{t}(\bar f)(\lambda_1&\bar{x} + \lambda_2 \bar{y}))^{1/n} \geq (1 - \lambda_1 - \lambda_2) \vol_{n}\bigl(\mathcal{A}^{\hid}_{t}(\bar f)(\bar o)\bigr)^{1/n}\\
        &+ \lambda_1 \vol_{n}\bigl(\mathcal{A}^{\hid}_{t}(\bar f)(\bar{x})\bigr)^{1/n} + \lambda_2 \vol_{n}\bigl(\mathcal{A}^{\hid}_{t}(\bar f)(\bar{y})\bigr)^{1/n}.
    \end{split}
\end{equation*}
Thus, taking into account that $\bar{x} \in {C}^{\hid}_{\theta_1,t}(\bar f), \bar{y} \in {C}^{\hid}_{\theta_2,t}(\bar f)$, we deduce that
\begin{equation*}
    \begin{split}
         \vol_n(\mathcal{A}^{\hid}_{t}(\bar f)(\lambda_1\bar{x} + \lambda_2 \bar{y}))^{1/n} &\geq (1 - \lambda_1 - \lambda_2) M_t^{1/n} + \lambda_1 \theta_1 M_t^{1/n} + \lambda_2 \theta_2 M_t^{1/n}\\
         &= \bigl( 1 - \lambda_1(1 - \theta_1^{1/n}) - \lambda_2(1 - \theta_2^{1/n})\bigr) M_t^{1/n},
    \end{split}
\end{equation*}
concluding the proof.
\end{proof}
\begin{remark}
    By taking $\theta_1=\theta_2=\theta_0$ and $\lambda_1+\lambda_2 = \frac{1-\theta^\frac{1}{n}}{1-\theta_0^\frac{1}{n}}$ in Lemma \ref{l:theta_convolution}, one obtains for every $0 \leq \theta_0 \leq \theta < 1$, that
    \begin{equation}\label{e:monotonicity_theta_convolution}
        \frac{\mathcal{C}^{\hid}_{\theta_0,t}(\bar f)}{1 - \theta_0^{1/n}} \subseteq \frac{\mathcal{C}^{\hid}_{\theta,t}(\bar f)}{1 - \theta^{1/n}}.
    \end{equation}
\end{remark}

We are now finally ready to prove the main results of this section. We start by providing the inequality for $\hid +1$ functions.
\begin{proof}[Proof of Theorem~\ref{t:hi_fun_RS}]   
    By homogeneity of the inequality, we may dilate the functions $f_i$ so that $\|f_i\|_\infty = 1$ for all $i=0,1,\dots,\hid$. We write $$C_{\theta,t}^\hid = C_{\theta,t}^\hid(\bar f).$$ 
    For any $t\in (0,1]$, let $$\bar x(t)=(x_1(t),\dots,x_m(t))\in \R^{nm}$$ be such that $M_t = \vol_n(\mathcal{A}^{\hid}_{t}(\bar f)(\bar{x}(t)))$.

Firstly, it follows from  \eqref{e:monotonicity_theta_convolution}, with $\theta_0=0$ and each $f_i$ in $\bar f$ replaced by $f_i(\cdot +x_i(t))$, that $$(1 - \theta^{1/n})(\mathcal{C}^m_{0,t} - \bar x(t)) \subset \mathcal{C}^m_{\theta,t}$$ for all $\theta \in [0,1]$. Thus, we have that
    $$
    \int_0^1 (1 - \theta^{1/n})^{nm}\vol_{nm}(\mathcal{C}^m_{0,t})\,\dlat \theta \leq \int_0^1 \vol_{nm}(\mathcal{C}^m_{\theta,t})\,\dlat \theta.
    $$
    Therefore, from the definition of $M_t$,
    \begin{equation*}
        \begin{split}
            \vol_{nm}(\mathcal{C}^m_{0,t}) &\leq \binom{n(\hid + 1)}{n} \int_0^1 \vol_{nm}(\mathcal{C}^m_{\theta,t})\,\dlat \theta\\
            & = \binom{n(\hid + 1)}{n} \int_{\R^{n\hid}} \frac{\vol_n\bigl(\mathcal{A}_t^m(\bar f)(\bar{x})\bigr)}{M_t} \, \dlat \bar{x}.
            \end{split}
    \end{equation*}
    Now, on the one hand, integrating with respect to the variable $t$, we deduce that 
    \begin{equation}
    \label{eq:up_m_t_bound}
        \begin{split}
            \int_0^1 M_t\,\vol_{nm}(\mathcal{C}^m_{0,t}) \, \dlat t &\leq \binom{n(\hid + 1)}{n} \int_0^1 \int_{\R^{n\hid}} \vol_n\bigl(\mathcal{A}_t^m(\bar f)(\bar{x})\bigr) \, \dlat \bar{x} \, \dlat t\\
            &= \binom{n(\hid + 1)}{n} \int_{\R^n} \int_{\R^{n\hid}} \!\!\!\!f_0(z) \prod_{i=1}^{m}f_i(z-x_i) \,\dlat\bar{x}\, \dlat z\\
            &= \binom{n(\hid + 1)}{n} \left(\prod_{i = 0}^\hid \int_{\R^n}f_i(x)\,\dlat x\right).
        \end{split}
    \end{equation}

    On the other hand, 
    \begin{equation}
    \label{eq:m_t_int}
        \begin{split}
            \int_0^1 &M_t\,\vol_{nm}(\mathcal{C}^m_{0,t}) \, \dlat t = \int_{\R^{n\hid}} \int_0^{(\bar{f})_{\star_{\hid}}(\bar{x})} \max_{\bar{y}} \vol_{n}\bigl(\mathcal{A}_t^m(\bar f)(\bar{y})\bigr)  \, \dlat t\,\dlat \bar{x}
            \\
            &\geq\max_{\bar{y}} \int_{\R^{n\hid}} \int_0^{(\bar{f})_{\star_{\hid}}(\bar{x})} \vol_{n}\bigl(\mathcal{A}_t^m(\bar f)(\bar{y})\bigr)  \, \dlat t \,\dlat \bar{x}\\
            &\geq \max_{\bar{y}} \int_{(\R^n)^{m+1}} \min \left\{(\bar{f})_{\star_{\hid}}(\bar{x}), f_0(z)\prod_{i=1}^m f_i(z-y_{i})\right\} \,\dlat z \,\dlat \bar{x},
        \end{split}
    \end{equation}
    where we used the definition of $\mathcal{A}_t^m(\bar f)(\bar{y})$. Recalling that $\|f_i\|_\infty = 1$ for all $i$, we see that we are considering in the final line of \eqref{eq:m_t_int} the minimum of functions less than $1$; this minimum is therefore greater than their product, and so our computation continues as
    \begin{equation}
    \label{eq:lower_m_t_bound}
        \begin{split}
            \int_0^1 M_t\,\vol_{nm}(\mathcal{C}^m_{0,t}) \, \dlat t &\geq \max_{\bar{y}} \int_{\R^n} \!\!\!f_0(z)\prod_{i=1}^m f_i(z-y_{i}) \,\dlat z \int_{\R^{n\hid}}\!\!\!\!\!(\bar{f})_{\star_{\hid}}(\bar{x})\dlat \bar{x}\\
            &= \|(\bar f)_{\oplus_{\hid}} \|_{\infty} \int_{\R^{n\hid}}(\bar{f})_{\star_{\hid}}(\bar{x})\dlat \bar{x}.
        \end{split}
    \end{equation}
    Combining \eqref{eq:up_m_t_bound} with \eqref{eq:lower_m_t_bound} concludes the proof.
\end{proof}

\begin{proof}[Proof of Theorem~\ref{t:hi_fun_RS_one}]
First, we establish that, for $t\in (0,1)$,
$$
\mathcal{C}^{\hid}_{0,t} \subset D^{m}\bigl(\mathcal{A}^{\hid}_{t^{m+1}}(f)(\bar o)\bigr).
$$
Indeed, let $\bar x \in \mathcal{C}^{\hid}_{0,t}$. Then, $\mathcal{A}^{\hid}_{t^{m+1}}(f)(\bar o)$ is non-empty. Thus, there exists $z\in\R^n$ such that $f(z)f(z - x_1) \cdots f(z - x_\hid) \geq t \|f\|_\infty^{m+1}.$ But, this means that $f(z), f(z - x_1),\ldots,f(z - x_\hid) \geq t \|f\|_\infty$. Setting $x_0=0$, we deduce that $f(z-x_i)^{m+1}\geq t^{m+1}\|f\|_\infty^{m+1}$ for all $i=0,1,\ldots,m.$ By definition, we then have that $z-x_i \in \mathcal{A}^{\hid}_{t^{m+1}}(f)(\bar o)$ for all $i$, or equivalently, 
\[
\bar x \in \mathcal{A}^{\hid}_{t^{m+1}}(f)(\bar o) \cap (\mathcal{A}^{\hid}_{t^{m+1}}(f)(\bar o)+x_1)\cap\cdots \cap (\mathcal{A}^{\hid}_{t^{m+1}}(f)(\bar o)+x_m).
\]
This is precisely that $\bar x \in D^m\left(\mathcal{A}^{\hid}_{t^{m+1}}(f)(\bar o)\right).$ Consequently, by using the $m$th-order Rogers-Shephard inequality \eqref{eq:RSell}, we obtain
$$
\vol_{nm}(\mathcal{C}^{\hid}_{0,t}) \leq \binom{n(\hid + 1)}{n} \vol_{n}(\mathcal{A}^{\hid}_{t^{m+1}}(f)(\bar o))^m.
$$

Now, by Minkowski's inequality for integrals, we infer
\begin{equation*}
    \begin{split}
        \int_0^1 \vol_{n}(\mathcal{A}^{\hid}_{t^{m+1}}(f)(\bar o))^m \,\dlat t &= \int_0^1 \left(\int_{\R^n} \chi_{_{\{f/\|f\|_\infty \geq t\}}}(z)\,\dlat z\right)^m \,\dlat t\\
        &\leq \left(\int_{\R^n} \left(\int_0^1 \chi_{_{\{f/\|f\|_\infty \geq t\}}}(z)\,\dlat t\right)^\frac{1}{m} \,\dlat z\right)^m\\
        &= \frac{1}{\|f\|_\infty} \left(\int_{\R^n} f(z)^\frac{1}{m}\,\dlat z\right)^m = \frac{\|f\|_\frac{1}{m}}{\|f\|_\infty}.
    \end{split}
\end{equation*}
Moreover, taking into account that $\mathcal{A}_t^m (f)(\bar{x}) \neq \emptyset$ if and only if $$t \leq (\bar{f})_{\star_{\hid}}(\bar{x})/\|f\|_\infty^{m+1},$$ we obtain
\begin{equation*}
    \begin{split}
        \int_0^1 \vol_{nm}(\mathcal{C}^{m}_{0,t}) \,\dlat t &= \int_0^1 \int_{\R^{nm}} \chi_{_{\{\bar{x}:\mathcal{A}_t^m (f)(\bar{x}) \neq \emptyset\}}} \, \dlat \bar{x} \,\dlat t\\
        &= \int_{\R^{nm}} \int_0^1 \chi_{_{\{t:\mathcal{A}_t^m (f)(\bar{x}) \neq \emptyset\}}}\,\dlat t \, \dlat \bar{x}\\
        &= \int_{\R^{nm}} \int_0^{(\bar{f})_{\star_{\hid}}(\bar x) /\|f\|_\infty^{m+1}} \dlat t \,\dlat \bar{x}\\
        &= \frac{1}{\|f\|_\infty^{m +1}}\int_{\R^{nm}} (\bar{f})_{\star_{\hid}}(\bar{x}) \,\dlat \bar{x},
    \end{split}
\end{equation*}
as desired. 
\end{proof}

\begin{proof}[Equality case in Theorem \ref{t:hi_fun_RS_one}]
    Fix $\hid \in \N \setminus \{ 1 \}$, as the $\hid=1$ case is already established (\cite[Lemma 5.1]{AlGMJV}).  On the one hand, if $f$ is the characteristic function of an $n$-dimensional simplex, \eqref{e:hi_fun_RS_one} becomes the $m$th-order Rogers-Shephard inequality \eqref{eq:RSell} for an $n$-dimensional simplex, which is precisely the equality case. 

    On the other-hand, equality in  \eqref{e:hi_fun_RS_one} implies equality in the following:
    \[
    \begin{split}
    \int_0^1\vol_{n\hid}(\mathcal{C}^{\hid}_{0,t})\,\dlat t&\leq \int_0^1 \vol\Bigl(D^{m}\bigl(\mathcal{A}^{\hid}_{t^{m+1}}(f)(\bar o))\bigr)\Bigr) \,\dlat t 
    \\
    &\leq \int_0^1 \binom{n(\hid + 1)}{n} \vol_n(\mathcal{A}^{\hid}_{t^{m+1}}(f)(\bar o)) \, \dlat t.
    \end{split}
    \]
    Therefore, there is equality in 
    $$
    \vol_{n\hid}(\mathcal{C}^{\hid}_{0,t}) \leq \vol\Bigl(D^{m}\bigl(\mathcal{A}^{\hid}_{t^{m+1}}(f)(\bar o)\bigr)\Bigr) \leq \binom{n(\hid + 1)}{n} \vol_n(\mathcal{A}^{\hid}_{t^{m+1}}(f)(\bar o))^m.
    $$
Using again the equality conditions in \eqref{eq:RSell} one has that, for almost every $t \in (0,1)$, $\mathcal{A}^{\hid}_{t^{m+1}}(f)(\bar o)$ is an $n$-dimensional simplex. We next use the fact that equality implies equality in our use of Minkowski's integral inequality. We obtain that the function $(z,t)\mapsto \chi_{\{f/\|f\|_\infty \geq t\}}(z)$ equals $h(t)q(z)$ for some measurable functions $h:(0,1]\mapsto\R_+$ and $q:\R^n\mapsto\R_+$ almost everywhere.

Since $f$ is upper semi-continuous and integrable, the set $\{f = \|f\|_\infty\}$ must be compact, and, since $f$ is $\log$-concave, the super-level sets of $f$ are convex and $f$ is continuous when restricted to the interior of its support. We deduce that the equality holds on $\R^n\times (0,1]$, in which case we must have $h(t)$ is a constant (merely pick a $z_0\in \{f=\|f\|_\infty\}$ to obtain $1=h(t)q(z_0)$ for all $t\in (0,1]$). Additionally this means, for a fixed $z$, the value of $\chi_{\{f/\|f\|_\infty \geq t\}}(z)$ is the same for all $t$. 

In terms of $f$, this means if $z$ is such that $f(z)\geq t\|f\|_\infty$ for some $t\in (0,1]$, then $ f(z)\geq t\|f\|_\infty$ for all $t\in (0,1]$. We deduce that $f(z)=\|f\|_\infty$ for all such $z$. Likewise, if $z$ is such that $f(z) < t\|f\|_\infty$ for some $t\in (0,1] \rightarrow f(z)< t\|f\|_\infty$ for all $t\in (0,1],$ and so $f(z)=0$. Therefore, $f$ partitions $\R^{n}$ into two sets: a convex body upon which $f(z)=\|f\|_\infty$, and the rest of $\R^n$. In particular, the latter is such that $f(z)=0$, i.e. $f$ is the characteristic function of a convex body. We now use that, for $t\in (0,1)$, $\mathcal{A}^{\hid}_{t^{m+1}}(f)(\bar o)=\{f/\|f\|_\infty \geq t\}=\{f = \|f\|_\infty\}$ to deduce that this convex body is an $n$-dimensional simplex.
\end{proof}

\appendix
\section{Some technical poofs}
In Section~\ref{sec:class}, there were three propositions whose proofs were technical and dry. We establish them now.

\begin{proof}[Proof of Proposition~\ref{p:increasing_mono}]
     Let $0<p<q$ be fixed such that $I_p(\psi)<\infty$ and $I_q(\psi)<\infty$. Assume, without loss of generality, that $\|\psi\|_\infty=1$. Let $a=I_p(\psi)$ and $\varphi(r)=pr^{p-1}(\psi(r)-1_{[0,a]}(r))$. Notice that $\varphi\le 0$ on $[0,a]$, $\varphi\ge 0$ on $[a,\infty)$ and $\int_0^{\infty}\varphi(r)\;\dlat r=0$. Thus
\[
I_q(\psi)^q-I_q(1_{[0,a]})^q=\frac{q}{p}\int_0^{\infty}r^{q-p}\varphi(r)\; \dlat r=\frac{q}{p}\int_0^{\infty}(r^{q-p}-a^{q-p})\varphi(r)\; \dlat r\ge 0,
\]
since the integrand is non negative on $\R_{+}$. We conclude that
\[
I_q(\psi)^q\ge I_q(1_{[0,a]})^q=a^q=I_p(\psi)^q.
\]
Clearly, there is equality if and only if $\psi=\chi_{[0,a]}$ almost everywhere.

 We next consider $-1<p<q<0$. Let $\varphi(r)=\psi(0)-\psi(1/r)$. From the fact that $\psi$ decays from its maximum at the origin to zero, we have $\|\psi\|_\infty=\|\varphi\|_\infty$. Then, for $p\in (-1,0)$, the following identity holds by applying a change of variables
 \begin{align*}
      I_p(\psi) &= \left(\frac{p}{\|\psi\|_\infty}\int_0^{\infty}(\psi(r)-\psi(0))\; r^{p-1 }\dlat r\right)^\frac{1}{p} 
      \\
      &=  \left(\frac{|p|}{\|h\|_\infty}\int_0^{\infty}\varphi(r)\; r^{|p|-1 }\dlat r\right)^{-\frac{1}{|p|}} = I_{|p|}(\varphi)^{-1}.
 \end{align*}
 Picking $p<q<0$, we have $0<|q|<|p|$. By applying the first part, we have $$I_p(\psi) =I_{|p|}(\varphi)^{-1} \leq I_{|q|}(\varphi)^{-1} = I_q(\psi),$$
 as claimed. The equality conditions are immediate.
\end{proof}

\begin{proof}[Proof of Proposition~\ref{p:mellin_finite_supp}]
    Begin by writing, for $p>0$,
\begin{equation}
\label{eq:mellin_limit_proto}
\left(p\Mel{\psi}{p}\right)^\frac{1}{p} = R \left(\int_0^R\psi(r)\, \frac{p r^{p-1}\dlat r}{R^p}\right)^\frac{1}{p}.
\end{equation}
Since $\psi\leq \|\psi\|_\infty = 1$ almost everywhere, it is easy to see from \eqref{eq:mellin_limit_proto} that
\[
\left(p\Mel{\psi}{p}\right)^\frac{1}{p} \leq R.
\]
On the other hand, by the definition of essential support, for every $\delta \in (0,R),$ the set
\[
A_\delta = \{r\in (R-\delta,R): \psi(r)>0\}
\]
has positive measure. Let $M_\delta = \operatorname{ess}\sup_{r \in A_\delta}\psi(r) \in (0,1]$ and let $m_\delta \in (0,M_\delta)$. Then, the set
\[
E_\delta = \left\{r\in (R-\delta,R):\psi(r)\geq m_\delta\right\}
\]
has positive measure. Therefore, 
\begin{equation}
\label{eq:mellin_limit_proto_2}
\begin{split}
\left(p\Mel{\psi}{p}\right)^\frac{1}{p} &\geq R\left(\int_{E_\delta} \frac{p m_\delta r^{p-1}\dlat r}{R^p}\right)^\frac{1}{p}
\\
&\geq R\left(\frac{p}{R^p}\vol_1(E_\delta)m_{\delta}\right)^\frac{1}{p}\left(R-\delta\right)^{1-\frac{1}{p}}
\\
&=R\left(\frac{p}{R-\delta}\vol_1(E_\delta)m_{\delta}\right)^\frac{1}{p}\left(1-\frac{\delta}{R}\right).
\end{split}
\end{equation}
For sufficiently small $\delta$, the term adjacent to $R$ on the right-hand side of \eqref{eq:mellin_limit_proto_2} will converge to $1$ as $p\to \infty$. We deduce the inequality
\[
\lim_{p\to\infty}\left(p\Mel{\psi}{p}\right)^\frac{1}{p} \geq R\left(1-\frac{\delta}{R}\right).
\]
By sending $\delta\to 0$, we deduce the claim.
\end{proof}

\begin{proof}[Proof of Proposition~\ref{p:limit_body}]
Observe that, from \eqref{eq:alt_mellin}
\begin{align*}\lim_{p\to(-1)^+}(1+p)^\frac{1}{p}\rho_{K_p(g)}(\theta)&=\lim_{p\to (-1)^+}\left((1+p)\int_0^{\infty}\left(-\frac{\partial}{\partial r}\frac{g(r\theta)}{\|g\|_\infty}\right)r^{p}\; \dlat r\right)^\frac{1}{p}.\end{align*}
Therefore, our claim follows if we can show that for every $\theta\in\s$, $$\varphi(r) = \left(-\frac{\partial}{\partial r}\frac{g(r\theta)}{\|g\|_\infty}\right), \quad r>0,$$ satisfies the hypotheses of Lemma~\ref{l:fractional_deriv}. 

Since $g$ is log-concave, there exists, for a fixed $\theta\in\s$, a convex function $V:\R_+\longrightarrow\R_+$ such that $g(r\theta)=e^{-V(r)}$ on $[0,\rho_{\supp(g)}(\theta)]$. For brevity, we write primes for the one-sided derivatives of $V$. Observe that $\frac{\partial}{\partial r}g(r\theta)=-g(r\theta)V^\prime(r)$, and therefore we can write
\[
\varphi(r) = \frac{g(r\theta)V^\prime(r)}{\|g\|_\infty}, \quad r>0.
\]
Since $g$ is log-concave and $V$ is convex, $\varphi$ is measurable. Since $g$ obtains its maximum at the origin, and therefore $V$ obtains its minimum at the origin, $V^\prime(r)\geq 0$ for $r>0$. We deduce $\varphi$ is non-negative. Recalling also that $V$ has one-sided derivatives everywhere on its domain, we have that $V^\prime(0)$ exists (since primes are denoting one-sided derivatives), and we extend $\varphi$ to $r=0$ by continuity: $\varphi(0)=V^\prime(0)$. It remains to show that  $\int_0^{\infty} t^{p_0} \varphi(t) \mathrm{d} t<\infty$ for some $p_0 \in(-1,0)$. We will show that any $p\in (-1,0)$ works.

Since $V$ is convex, its different quotients are monotonic: fix $\delta>0$ so that $g(r\theta)$ is positive on $(0,\delta]$. Then, for every $r\in(0,\delta] $, we have
\[
0 \leq \frac{V(r)-V(0)}{r} \leq \frac{V(\delta)-V(0)}{\delta},
\]
where we used that $V(r)-V(0)\geq 0$ is positive, since $V$ attains its minimum at the origin. Sending $r\to 0^+$, we deduce
\[
0\leq V^\prime(0) \leq \frac{V(\delta)-V(0)}{\delta},
\]
which is finite. Since $V$ is convex, $V^\prime$ is increasing, therefore, using again that $V(0)$ is the minimum of $V$, we have for almost all $r\in (0,\delta],$ $$V^\prime(0) \leq V^\prime(r) \leq \frac{V(\delta)-V(0)}{\delta},$$ which yields that $V^\prime$ is bounded from above and below on $(0,\delta]$, and, therefore, so too is $\varphi.$ Thus, for every $p\in (-1,0)$,
$$\int_0^{\delta}\varphi(r)r^p  \; \dlat r < \infty.$$
On $(\delta,\infty)$, we use by Proposition~\ref{p:equivalence} that an integrable, log-concave function is bounded by an exponential; in particular, it decays to zero. For almost every $R>\delta$, we have
\begin{align*}
    \int_{\delta}^R \left(-\frac{\partial}{\partial r}g(r\theta)\right)  \; \dlat r  = g(\delta \theta)-g(R\theta).
\end{align*}
Taking the limit as $R\to\infty$, we deduce from Fatou's lemma
\[
\int_{\delta}^{\infty} \left(-\frac{\partial}{\partial r}g(r\theta)\right)  \; \dlat r  \leq g(\delta \theta).
\]
Therefore, using that $r^p \leq \delta^p$ on $(\delta,\infty)$ since $p\in (-1,0)$,
\[
\int_{\delta}^{\infty} \varphi(r)  r^p\; \dlat r  \leq \delta^p \frac{g(\delta \theta)}{\|g\|_\infty}.
\]
We conclude the proof.
\end{proof}

{\bf Acknowledgments:} This work was initiated while all three authors were in-residence at the Hausdorff Research Institute for Mathematics in Bonn, Germany, during Spring 2024 for the Dual Trimester Program: "Synergies between modern probability, geometric analysis and stochastic geometry". The three authors express their sincere gratitude to the Hausdorff Institute, its dedicated staff, and the organizers of the trimester program for their invaluable support and hospitality.
\medskip

{\bf Funding:} The first-named author was supported by the U.S. National Science Foundation's MSPRF fellowship via NSF grant DMS-2502744. The second named author is funded by the grant PID2021-124157NB-I00, financed by MCIN/ AEI/10.13039 /501100011033/ ``ERDF A way of making Europe'', and by Comunidad Autónoma de la Región de Murcia a través de la convocatoria de Ayudas a proyectos para el desarrollo de investigación científica y técnica por grupos competitivos, incluida en el Programa Regional de Fomento de la Investigación Científica y Técnica (Plan de Actuación 2022) de la Fundación Séneca-Agencia de Ciencia y Tecnología de la Región de Murcia, REF. 21899/PI/22. The third-named author was supported by the Austrian Science Fund (FWF): 10.55776/P34446, and, in part, by the Gruppo Nazionale per l'Analisi Matematica, la Probabilit\'a e le loro Applicazioni (GNAMPA) of the Istituto Nazionale di Alta Matematica (INdAM). 

\bibliographystyle{acm}
\bibliography{references}

\end{document}